\documentclass[11pt]{article}

\usepackage{graphicx}
\RequirePackage[OT1]{fontenc}
\RequirePackage{amsthm,amsmath}

\numberwithin{equation}{section}
\theoremstyle{plain}
\newtheorem{thm}{Theorem}

\newtheorem{lem}{Lemma}

\newcommand{\wtd}{\widetilde}

\begin{document}

\begin{center}
{\bf THRESHOLDING THE HIGHER CRITICISM TEST STATISTIC FOR OPTIMALITY IN A HETEROGENEOUS SETTING}

\medskip 
Hock Peng Chan \\
\it{Department of Statistics and Data Science} \\
\it{National University of Singapore}
\end{center}

\begin{abstract}
Donoho and Kipnis (2022) showed that the the higher criticism (HC) test statistic has
a non-Gaussian phase transition but remarked that it is probably not optimal, 
in the detection of sparse differences between two large frequency tables when the counts are low.
The setting can be considered to be heterogeneous, 
with cells containing larger total counts more able to detect smaller differences.
We provide a general study here of sparse detection arising from such heterogeneous settings,
and showed that optimality of the HC test statistic requires thresholding,
for example in the case of frequency table comparison,
to restrict to p-values of cells with total counts exceeding a threshold.
The use of thresholding also leads to optimality of the HC test statistic when it is applied on the sparse Poisson means model
of Arias-Castro and Wang (2015). 
The phase transitions we consider here are non-Gaussian, 
and involve an interplay between the rate functions of the response and sample size distributions.
We also showed, both theoretically and in a numerical study,
that applying thresholding to the Bonferroni test statistic results in better sparse mixture detection in heterogeneous settings. 
\end{abstract}

\section{Introduction}

The higher criticism (HC) test statistic has lately been studied under more
varied settings, for example in Arias-Castro, Cand\'{e}s and Plan (2011), 
Arias-Castro and Wang (2015), Cai and Wu (2014), Chan and Walther (2015), Donoho and Jin (2015),
Donho and Kipnis (2021, 2022), Jin, Ke and Wang (2016), Kipnis (2022),
Li and Siegmund (2015), Moscovich, Nadler and Spiegelman (2016)
and Mukherjee, Pillai and Lin (2015),
extending beyond the classical sparse Gaussian mixture considered in Donoho and Jin (2004) and Ingster (1997).
However asymptotics and optimality of the HC test statistic obtained have largely been confined to 
Gaussian phase transition, even when the underlying response distribution is non-Gaussian.
Notable exceptions are Donoho and Kipnis (2021, 2022),
where non-Gaussian phase transitions of the HC test statistic were obtained.
However it was unclear whether these phase transitions are optimal over all test statistics. 

In this paper we extend the optimality of the HC test statistic first to non-Gaussian responses,
with phase transitions that are unique to the response distributions.
This is followed by an extension to heterogeneous settings where there are differences in the detection powers of the local hypotheses.
For example in the comparison of two frequency tables, 
cells with larger total counts have larger detection powers for a fixed probability difference between the two tables.
In sparse Poisson means models, 
hypotheses with larger Poisson means have more detection powers, 
when the mean ratio between the null and alternative hypotheses is fixed.
 
The HC test statistic is however not optimal when applied directly to such heterogeneous settings.
This is remedied by thresholding the HC test statistic,
that is by restricting the p-values considered to say cells with total counts exceeding a threshold.
To be specific we derived an optimal phase transition curve that involves the response and sample size distributions, 
and showed that the threshold HC test statistic achieves this curve.
We also applied thresholding to the Bonferroni test statistic and showed, 
both theoretically and in a simulation study, 
that we are able to detect sparse mixtures better with thresholding. 

\subsection{Layout}

In Section 2 we provide a short recap of the HC test statistic and its Gaussian phase transition.
In Section 3 we extend to non-Gaussian phase transitions in homogeneous settings.
In Section 4 we extend further to phase transitions in heterogeneous settings,
and show that the threshold HC test statistic is optimal.
In Section 5 we introduce the rank-adjustment test statistic 
and provide the phase transitions of the Bonferroni and rank-adjustment test statistics.
In Section 6 we describe the numerical performances of all the test statistics.
The proofs of all results are given in Chan (2023).

\subsection{Notations}

We write $X \sim F$ to denote random variable $X$ following a distribution $F$.
We write {\it i.i.d.} to denote independent and identically distributed.
We write $F^k$ to denote the $k$-fold convolution of a distribution $F$.
That is $F^k$ is the distribution of $Y=X_1+\cdots+X_k$ when $X_i \sim_{\rm i.i.d.} F$.
We write $P_{\theta}$ and $E_{\theta}$ to denote probability and expectation with respect to a
distribution $F_{\theta}$.
We write $P_0$ ($E_0$) and $P_1$ ($E_1$) to denote probability (expectation) with respect to the global null and alternative
respectively.
For sequences $(a_n)$ and $(b_n)$ we write $a_n \sim b_n$ to denote $\tfrac{a_n}{b_n} \rightarrow 1$,
and $a_n  =o(b_n)$ to denote $\tfrac{a_n}{b_n} \rightarrow 0$.
We write $\# A$ to denote the number of elements in a set $A$.
For a given function $\psi$ we write $\psi'$ to denote the first derivative of $\psi$. 

\section{Gaussian phase transition of the HC test statistic}

Consider random variables $X_{ij}$, $1 \leq i \leq n$, $1 \leq j \leq k_n$,
following an exponential family of distributions $\{ F_{\theta} \}_{\theta \in \Theta}$,
satisfying
\begin{equation} \label{expo}
dF_{\theta}(x) = e^{\theta x-\psi(\theta)} dF_0(x),
\end{equation}
where $e^{\psi(\theta)} = E_0 e^{\theta X}$ and $\Theta = \{ \theta: \psi(\theta) < \infty \}$.

Consider a global null $H_0$ under which $X_{ij} \sim_{\rm i.i.d.} F_{\theta_0}$.
Since $\theta_0$ is known we may assume, by a reparametrization if necessary,
that $\theta_0=0$.
Let $Y_{in} = X_{i1} + \cdots + X_{ik_n}$ and consider the global alternative $H_1$
under which,
for some unknown $\theta \neq 0$ and $\epsilon_n > 0$,
\begin{equation} \label{mixture}
Y_{in} \sim_{\rm i.i.d.} (1-\epsilon_n) F_0^{k_n} + \epsilon_n F_{\theta}^{k_n}.
\end{equation}

Let $p_i$ be the p-value of the $i$th hypotheses $H_{0i}$: $Y_{in} \sim F_0^{k_n}$ 
and $H_{1i}$: $Y_{in} \sim F_{\theta}^{k_n}$, and let
$p_{(1)} \leq \cdots \leq p_{(n)}$ be the sorted p-values.
We show in Theorem \ref{thm1} that Tukey's (1976) HC test statistic,
\begin{equation} \label{HC}
{\rm HC}_n = \max_{1 \leq i \leq \frac{n}{2}} \tfrac{i - np_{(i)}}{\sqrt{np_{(i)} (1-p_{(i)})}},
\end{equation}
is optimal in the detection of the global null versus alternative,
in the sparse mixture setting when $\epsilon_n = n^{-\beta}$ under $H_1$,
and $\tfrac{k_n}{\log n} \rightarrow a$ for some $a > 0$.

When the underlying exponential family is Gaussian this is the celebrated result
in Donoho and Jin (2004) that sparked interest and subsequent developments in the study of sparse mixture detection
using the HC test statistic.
Theorem \ref{thm1} extends optimality of the HC test statistic to general exponential families.

Let the risk of a test statistic $T$ be given by
\begin{equation} \label{risk}
\mbox{Risk}(T) = \inf_c [P_0(T \geq c) + P_1(T < c)].
\end{equation}
If $T$ is stochastically smaller under $H_1$ compared to $H_0$,
the inequalities are reversed in (\ref{risk}).
Following Arias-Castro, Cand\'{e}s and Plan (2011) and Donoho and Kipnis (2021),
we say that $T_n$ is asymptotically powerful if Risk($T_n) \rightarrow 0$ and asymptotically powerless if 
Risk($T_n) \rightarrow 1$.

To motivate the non-Gaussian phase transition curves in the statement of Theorem \ref{thm1},
we first look at the phase transition curve under the Gaussian setting.
Consider $F_{\theta} = {\rm N}(\theta,1)$.
If $X_{ij} \sim_{\rm i.i.d.} F_{\theta}$ for $1 \leq j \leq k_n$
then $k_n^{-\frac{1}{2}} Y_{in} \sim {\rm N}(\theta \sqrt{k_n},1)$.
Consider $\tfrac{k_n}{\log n} \rightarrow a$ for some $a>0$.
By Donoho and Jin (2004) and Ingster (1997), 
the Gaussian phase transition curve
\begin{equation} \label{rho}
\rho(\beta) = \left\{ \begin{array}{ll} \beta-\tfrac{1}{2} & \mbox{ if } \tfrac{1}{2} < \beta \leq \tfrac{3}{4}, \cr
(1-\sqrt{1-\beta})^2 & \mbox{ if } \frac{3}{4} < \beta < 1, \end{array} \right.
\end{equation}
is such that in the detection of the sparse component of  
$(1-\epsilon_n) {\rm N}(0,1)+ \epsilon_n {\rm N}(\theta \sqrt{k_n},1)$,
all test statistics are asymptotically powerless when $|\theta| \sqrt{a} < \sqrt{2 \rho(\beta)}$,
whereas the HC test statistic is asymptotically powerful when $|\theta| \sqrt{a} > \sqrt{2 \rho(\beta)}$.

Whereas the chi-squared test is known to be a powerful test under the dense regime 
$\beta < \tfrac{1}{2}$, it is a weak test when the sparse regime $\tfrac{1}{2} < \beta < 1$,
so the HC test statistic complements the chi-squared test statistic by achieving
the optimal phase transition curve (\ref{rho}) when $\tfrac{1}{2} < \beta < 1$.

For non-Gaussian $F_0$ it is neater to express the phase transition curve as a function of $\theta$.
That is we want a phase transition curve $b(\theta)$ such that the HC test statistic is asymptotically powerful when
$\beta < b(\theta)$,
whereas all test statistics are asymptotically powerless when $b(\theta) < \beta$.

By (\ref{rho}) for $F_0$ standard normal, we can express this phase transition curve as
$$b(\theta) = \left\{ \begin{array}{ll} 
\tfrac{1}{2}(1+a \theta^2) & \mbox{ if } a \theta^2 \leq \tfrac{1}{2}, \cr
1-(1-\sqrt{\tfrac{a \theta^2}{2}})^2 & \mbox{ if } \tfrac{1}{2} < a \theta^2 \leq 2. \end{array} \right.
$$

\section{Non-Gaussian phase transitions of the HC test statistic}

Let $I(\nu) = \sup_{\theta \in \Theta} [\theta \nu-\psi(\theta)]$ be the Legendre-Fenchel transform of $\psi(\theta)$,
that is the rate function of $F_0$.
By Cram\'{e}r's Theorem for $\nu > \mu(0)$,
$$P_0 (Y_{in} \geq k_n \nu) = e^{-k_n [I(\nu)+o(1)]} \mbox{ as } k_n \rightarrow \infty.
$$
More generally for $\nu > \mu(\theta)$,
\begin{equation} \label{cram2}
P_{\theta} (Y_{in} \geq k_n \nu) = e^{-k_n[I(\nu)-\theta \nu+\psi(\theta)+o(1)]} \mbox{ as } k_n \rightarrow \infty,
\end{equation}
and (\ref{cram2}) holds with the inequality reversed when $\nu < \mu(\theta)$.

Let $\theta_a^+ > 0$ and $\theta_a^- <0$ be such that $\mu_a^+ = \psi'(\theta_a^+)$ and $\mu_a^-=\psi'(\theta_a^-)$ satisfy
\begin{equation} \label{Imu} 
I(\mu_a^+) = I(\mu_a^-) = a^{-1}. 
\end{equation}
It is possible when $F_0$ is discrete and bounded, 
for example when $F_0$ is Bernoulli or Poisson,
that there is no solution $\mu_a^+$ or $\mu_a^-$ to (\ref{Imu}).

To take care of these special cases we can define more generally
$$\mu_a^+ = \sup \{ \mu: I(\mu) \leq a^{-1} \}, \qquad \mu_a^- = \inf \{ \mu: I(\mu) \leq a^{-1} \},
$$
with $\theta_a^+$ (possibly $\infty$) and $\theta_a^-$ (possibly $-\infty$) satisfying
$\mu_a^+ = \lim_{\theta \rightarrow \theta_a^+} \psi'(\theta)$ 
and $\mu_a^- = \lim_{\theta \rightarrow \theta_a^-} \psi'(\theta)$.

The phase transition curve of $F_0$ is 
\begin{equation} \label{btheta}
b(\theta) = \left\{ \begin{array}{ll} \tfrac{1}{2} \{1+a[\psi(2 \theta)-2 \psi(\theta)] \} & 
\mbox{ if } \tfrac{1}{2} \theta_a^- \leq \theta \leq \tfrac{1}{2} \theta_a^+, \cr 
a[\theta \mu_a^+ - \psi(\theta)] & \mbox{ if } \tfrac{1}{2} \theta_a^+ < \theta \leq \theta_a^+, \cr 
a[\theta \mu_a^- - \psi(\theta)] & \mbox{ if } \theta_a^- \leq \theta < \tfrac{1}{2} \theta_a^-. 
\end{array} \right.
\end{equation}

Let $p_i$ be the two-sided p-value of $Y_{in}=X_{i1}+ \ldots +X_{i k_n}$.
That is $p_i$ is two times the upper or lower tail probability of $Y_{in} \sim F_0^{k_n}$,  
whichever is smaller. 
For $F_0$ discrete we randomize the p-values so that $p_i \sim {\rm Uniform}(0,1)$ under the $i$th null $H_{0i}$. 
The randomization is for theoretical convenience. 
In practice we may want to apply the HC test statistic without p-value randomization.
Similar p-value randomizations were applied in Donoho and Kipnis (2022).

\begin{thm} \label{thm1} 
Consider the sparse testing problem {\rm (\ref{mixture})} with $\epsilon_n=0$ under $H_0$
and $\epsilon_n = n^{-\beta}$ for some $0 < \beta < 1$ under $H_1$.
Consider $\tfrac{k_n}{\log n} \rightarrow a$ for some $a >0$ as $n \rightarrow \infty$. 
If $\beta>b(\theta)$ then all test statistics are asymptotically powerless. 
If $\beta < b(\theta)$ then the HC test statistic is asymptotically powerful. 
\end{thm}

{\sc Example 1}. Consider the sparse mixture detection of 
\begin{equation} \label{YPoisson}
Y_{in} \sim_{i.i.d.} (1-\epsilon_n) {\rm Poisson}(k_n) + \epsilon_n {\rm Poisson}(e^{\theta} k_n), \quad 1 \leq i \leq n,
\end{equation}
for some $\theta \neq 0$,
with $\tfrac{k_n}{\log n} \rightarrow a$ for some $a>0$.

When $k_n$ is integer-valued, 
this falls under the setting of Theorem \ref{thm1} with $F_0$ the Poisson
distribution with mean 1. 
The restriction of $k_n$ to integer values is only due to the setting of the problem
considered in Theorem \ref{thm1}.
The proof of Theorem \ref{thm1} easily extends to non integer-valued $k_n$ under the Poisson setting.

When $F_0$ is the Poisson(1) distribution, 
$\psi(\theta)=e^{\theta}-1$ and its rate function is $I(\mu) = \mu \log \mu-\mu+1$.
Hence by (\ref{btheta}),
the phase transition curve is 
$$b(\theta) = \left\{ \begin{array}{ll} \tfrac{1}{2}[1+a(e^{2 \theta}-2e^{\theta}+1)] & \mbox{ if } 
\tfrac{1}{2} \theta_a^- \leq \theta \leq \tfrac{1}{2} \theta_a^+, \cr
a[\theta \mu_a^+ - e^{\theta}+1] & \mbox{ if } \tfrac{1}{2} \theta_a^+ < \theta \leq \theta_a^+, \cr
a[\theta \mu_a^- - e^{\theta}+1] & \mbox{ if } \theta_a^- \leq \theta < \tfrac{1}{2} \theta_a^-, \cr
\end{array} \right.
$$
where $\theta_a^-<0$ and $\theta_a^+>0$ are the solutions of
\begin{equation} \label{etheta}
e^y(y-1)+1=a^{-1}. 
\end{equation}
Since $e^y(y-1)+1 \uparrow 1$ as $y \rightarrow -\infty$, 
when $a \leq 1$ there is no negative solution of (\ref{etheta}) and so $\theta_a^-=-\infty$ and $\mu_a^-=0$.

Theorem 1 thus fills the gap between the high count ($\tfrac{k_n}{\log n} \rightarrow \infty$), 
and low count ($\tfrac{k_n}{\log n} \rightarrow 0$) asymptotics in Arias-Castro and Wang (2011),
by showing that the HC test statistic is optimal under the intermediate setting $\tfrac{k_n}{\log n} \rightarrow a$
for some $a>0$.

The setting in Arias-Castro and Wang (2011) is more general than what is described above,
with the mean of $Y_{in}$ under $H_0$ possibly varying with $i$.
This brings us to the study, in the next section, of the phase transitions for heterogeneous settings,
and the thresholding of the HC test statistic to achieve optimality.

\section{Phase transitions for heterogeneous settings}

The phase transition curves for heterogeneous settings has an added level of complexity
as they involve the rate function of the sample size distribution as well as the rate function of the response distribution. 

Consider random variables $X_{ij}$, $1 \leq i \leq n$, $1 \leq j \leq K_{in}$
and let $Y_i = \sum_{j=1}^{K_{in}} X_{ij}$.
Let $K_{in}$, $1 \leq i \leq n$, be i.i.d. with rate function $J$ satisfying the following conditions.

\smallskip \noindent
(A1) There exists $a_0>0$ such that $J(a)=0$ for $0 \leq a \leq a_0$.

\smallskip \noindent
(A2) $J'(a)$ is continuous and strictly increasing on $[a_0,\infty)$,
with $J'(a_0)=0$ and $\lim_{a \rightarrow \infty} J'(a)=
\infty$.

\smallskip \noindent
(A3) $P(K_{in} = k) = \alpha_{kn} n^{-J(\frac{k}{\log n})}$, 
with $\sup_{k \geq \lambda_n} \tfrac{|\log \alpha_{kn}|}{k}
\rightarrow 0$ as $n \rightarrow \infty$,
with $\lambda_n = a_0 \log n$.

\medskip
{\sc Example} 2. (a) Consider $K_{in} \sim \max(1,{\rm Poisson}(\lambda_n))$.
Let $a=\tfrac{k}{\log n}$.
By Stirling's approximation, 
for $k \geq \lambda_n$,
$$P(K_{in} =k) = e^{-\lambda_n} \tfrac{\lambda_n^k}{k!} \sim \tfrac{1}{\sqrt{2 \pi k}} e^{-\lambda_n}
(\tfrac{\lambda_n e}{k})^k = \tfrac{1}{\sqrt{2 \pi k}} e^{-a_0 \log n} (\tfrac{a_0 e}{a})^{a \log n},
$$
and (A1)--(A3) holds with $J(a) = a \log(\tfrac{a}{a_0})-a+a_0$ for $a \geq a_0$.

\smallskip \noindent
(b) Consider
$$P(K_{in} = k) = \int_k^{k+1} \tfrac{1}{\sqrt{2 \pi \tau \lambda_n}} e^{-\tfrac{(z-\lambda_n)^2}{2 \tau \lambda_n}} dz,
\qquad k \geq \lambda_n,
$$
corresponding to an asymptotic ${\rm N}(\lambda_n,\tau \lambda_n)$ distribution.
We can check that (A1)--(A3) hold with $J(a) = \tfrac{(a-a_0)^2}{2a_0 \tau}$ for $a \geq a_0$.

\medskip  
We express the phase transition curve $b_J(\theta)$ for a heterogeneous setting in terms of
the following constrained optimization problem.
We consider only $\theta$ such that $a_0 I(\mu(\theta)) \leq 1$, 
since a simple Bonferroni test is asymptotically powerful for $0 < \beta < 1$ when $a_0 I(\mu(\theta)) > 1$.
Let
\begin{eqnarray} \label{gnu}
g(\nu,a) & = & a I(\nu) + J(a), \\ \label{fnu}
f_{\theta}(\nu,a) & = & a[\theta \nu-\psi(\theta)]+\tfrac{1}{2}[1-g(\nu,a)].
\end{eqnarray}
For $\theta$ such that $a_0 I(\mu(\theta)) \leq 1$,
define
\begin{equation} \label{bJ}
b_J(\theta) = \max_{(\nu,a): g(\nu,a) \leq 1} f_{\theta}(\nu,a).
\end{equation}

\medskip 
{\sc Remarks}.
Without the constraint $g(\nu,a) \leq 1$,
maximization of $f_{\theta}(\nu,a)$ occurs at $(\mu(2 \theta),a_{\theta})$,
with $a_{\theta}$ satisfying
\begin{equation} \label{J1}
J'(a_{\theta}) = \psi(2 \theta)-2 \psi(\theta).
\end{equation}
Hence for $\theta$ such that $g(\mu(2 \theta),a_{\theta}) \leq 1$,
$$b_J(\theta) = f_{\theta}(\mu(2 \theta),a_{\theta}) = \tfrac{1}{2}[1+a_{\theta} J'(a_{\theta})-J(a_{\theta})].
$$

For $\theta$ such that $g(\mu(2 \theta),a_{\theta})>1$,
maximization of $f_{\theta}(\nu,a)$ in (\ref{bJ}) occurs at $g(\nu,a)=1$.
In particular by the method of Lagrange multipliers,
$$b_J(\theta) = f_{\theta}(\mu(\theta^*),a_{\theta}^*),
$$
with $(\theta^*,a_{\theta}^*)$ characterized by
\begin{eqnarray} \label{J2a}
J'(a_{\theta}^*) & = & \psi(\theta^*)-\tfrac{\theta^*}{\theta} \psi(\theta), \\ \label{J3}
g(\mu(\theta^*),a_{\theta}^*) & = & 1, 
\end{eqnarray}
and $1 \leq \tfrac{\theta^*}{\theta} < 2$. 

\medskip
{\sc Example} 3. (a) Consider $J(a) =a \log (\tfrac{a}{a_0})-a+a_0$ for $a \geq a_0$,
the rate function of ${\rm Poisson}(\lambda_n)$.
Since $J'(a) = \log (\tfrac{a}{a_0})$ for $a \geq a_0$,
by (\ref{J1}), 
$$a_{\theta} = a_0 e^{\psi(2 \theta)-2 \psi(\theta)}.
$$
Hence for $\theta$ such that $g(\mu(2 \theta)),a_{\theta}) \leq 1$, 
$$b_J(\theta) = \tfrac{1}{2} \{ 1+a_0 [e^{\psi(2 \theta)-2 \psi(\theta)}-1] \}.
$$

\smallskip
\noindent (b) Consider $J(a) = \tfrac{(a-a_0)^2}{2a_0 \tau}$ for $a \geq a_0$, 
the rate function of ${\rm N}(\lambda_n, \tau \lambda_n)$.
Since $J'(a)=\tfrac{a-a_0}{a_0 \tau}$ for $a \geq a_0$, 
by (\ref{J1}), 
$$a_{\theta} = a_0 \{ 1+\tau[\psi(2 \theta)-2 \psi(\theta)] \}. 
$$
Hence for $\theta$ such that $g(\mu(2 \theta)),a_{\theta}) \leq 1$,
$$b_J(\theta) = \tfrac{1}{2} \{ 1 + a_0[\psi(2 \theta)-2 \psi(\theta) + \tfrac{\tau [\psi(2 \theta)-2 \psi(\theta)]^2}{2} ]\}.
$$

\subsection{Optimality of the HC test statistic via thresholding}

The HC test statistic does not attain the optimal phase transition curve $b_J(\theta)$.
To achieve this detection boundary we need to threshold the HC test statistic.
For a given $k \geq 1$, 
let ${\rm HC}_{kn}$ be the HC test statistic
computed on the p-values $p_i$ for $K_{in} \geq k$.

That is let $A_k = \{ i: K_{in} \geq k \}$, $n_k = \# A_k$ and let $p_{(i)k}$
the the $i$th smallest p-value among $\{ p_i: i \in A_k \}$.
Define
$${\rm HC}_{kn} = \max_{1 \leq i \leq \frac{n_k}{2}} \tfrac{i-n_k p_{(i)k}}{\sqrt{n_k p_{(i)k} (1-p_{(i)k})}}.
$$
The threshold HC test statistic is defined to be
\begin{equation} \label{HCthres}
{\rm HC}_n^{\rm thres} = \max_{k \geq 1} {\rm HC}_{kn}.
\end{equation}

We show in Theorem \ref{thm2} that ${\rm HC}_n^{\rm thres}$ is optimal in detecting a sparse mixture for 
heterogeneous settings.
However the computation of ${\rm HC}_n^{\rm thres}$ is expensive compared to that of ${\rm HC}_n$.
In practice to reduce computation cost we may want to maximize HC$_{kn}$ in (\ref{HCthres}) over a few representative
values of~$k$.

For example it can be shown that when the underlying exponential family is Gaussian,
the maximization in (\ref{bJ}) is achieved when $J(a) \leq \tfrac{1}{2}$.
This motivates the maximization in (\ref{HCthres}) to $k$ satisfying $n_k \geq n^{\frac{1}{2}}$.

\begin{thm} \label{thm2}
Consider the test of $H_0$: $\epsilon_n=0$ versus $H_1$: $\epsilon=n^{-\beta}$ for some $0 < \beta < 1$, 
with sample sizes $K_{in}$ i.i.d. with rate function $J$ satisfying {\rm (A1)--(A3)}.
If $\beta > b_J(\theta)$ then all test statistics are asymptotically powerless. 
If $\beta < b_J(\theta)$ then ${\rm HC}^{\rm thres}_n$ is asymptotically powerful.
\end{thm}

\subsection{Phase transition of the HC test statistic}

For completeness we characterize the phase transition curve $b_H(\theta)$ for ${\rm HC}_n$,
with no thresholding of sample sizes.
As in the case of the threshold HC test statistic,
the phase transition curve is expressed as a constrained optimization problem.
In Donoho and Kipnis (2021),
an impossibility region in which HC$_n$ is asymptotically powerless was obtained using constrained optimization.
The setting there is more general than what is considered in Theorem \ref{thm3}.

Let
\begin{eqnarray} \label{gH}
g_H(\nu,a) & = & a I(\nu) + 2 J(a), \\ \label{fH}
f_{H \theta}(\nu,a) & = & a[\theta \nu-\psi(\theta)]+\tfrac{1}{2}[1-g_H(\nu,a)].
\end{eqnarray}
For $\theta$ such that $a_0 I(\mu(\theta)) \leq 1$,
define
\begin{equation} \label{bH}
b_H(\theta) = \max_{(\nu,a): a I(\nu) \leq 1} f_{H \theta}(\nu,a).
\end{equation}

\medskip
{\sc Remarks}.
Without the constraint $aI(\nu) \leq 1$,
maximization of $f_{H \theta}(\nu,a)$ occurs at $(\mu(2 \theta),a_{H \theta})$, 
with $a_{H \theta}$ satisfying
\begin{equation} \label{J2}
J'(a_{H \theta}) = \tfrac{\psi(2 \theta)}{2}-\psi(\theta).
\end{equation}
Hence for $\theta$ such that $a_{H \theta} I(\mu(2 \theta)) \leq 1$,
$$b_H(\theta) = f_{H \theta}(\mu(2 \theta),a_{H \theta}) = \tfrac{1}{2} + a_{H \theta} J'(a_{H \theta}) - J(a_{H \theta}).
$$

For $\theta$ such that $a_{H \theta} I(\mu(2 \theta)) > 1$,
maximization of $f_{H \theta}(\nu,a)$ in (\ref{bH}) occurs at $a I(\nu)=1$.
In particular by the method of Lagrange multipliers,
$$b_H(\theta) = f_{H \theta}(\mu(\theta_H^*),a_{H \theta}^*),
$$
with $(\theta_H^*,a_{H \theta}^*)$ characterized by
\begin{eqnarray} \label{J1a}
J'(a_{H \theta}^*) & = & \tfrac{\theta}{\theta_H^*} \psi(\theta_H^*)-\psi(\theta), \\ \label{astar}
a_{H \theta}^* I(\mu(\theta_H^*)) & = & 1, 
\end{eqnarray}
and $1 \leq \tfrac{\theta_H^*}{\theta} < 2$.

\medskip
{\sc Example} 4. (a) 
Consider $J(a) = a \log(\tfrac{a}{a_0})-a+a_0$ for $a \geq a_0$,
the rate function of Poisson($\lambda_n$). 
Since $J'(a) = \log(\tfrac{a}{a_0})$ for $a \geq a_0$,
by (\ref{J2}),
\begin{equation} \label{aH2}
a_{H \theta} =a_0 \exp(\tfrac{\psi(2 \theta)}{2}-\psi(\theta)).
\end{equation}
Hence for $\theta$ such that $a_{H \theta} I(\mu(2 \theta)) \leq 1$,
$$b_H(\theta) = \tfrac{1}{2}+a_0[\exp(\tfrac{\psi(2 \theta)}{2}-\psi(\theta))-1].
$$

\medskip \noindent
(b) Consider $J(a) = \tfrac{(a-a_0)^2}{2a_0 \tau}$ for $a \geq a_0$,
the rate function of ${\rm N}(\lambda_n,\tau \lambda_n)$.
Since $J'(a)=\tfrac{a-a_0}{a_0 \tau}$ for $\tau \geq a_0$,
by (\ref{J2}),
$$a_{H \theta} = a_0 \{ 1+\tfrac{\tau}{2}[\psi(2 \theta)-2 \psi(\theta)] \}.
$$
Hence for $\theta$ such that $a_{H \theta} I(\mu(2 \theta)) \leq 1$,
$$b_H(\theta) = \tfrac{1}{2} \{ 1+a_0 [\psi(2 \theta)-2 \psi(\theta) + \tfrac{\tau[\psi(2 \theta)-2 \psi(\theta)]^2}{4}] \}.
$$

\begin{thm} \label{thm3}
Consider the sparse mixture problem in Theorem {\rm 2}.
The test statistic ${\rm HC}_n$ is asymptotically powerless when 
$\beta > b_H(\theta)$ and asymptotically powerful when $\beta < b_H(\theta)$. 
\end{thm}

\medskip
{\sc Example} 5.
For low frequency table counts,
the sparse mixture problem considered in Donoho and Kipnis (2022) corresponds to 
$F_0={\rm Bernoulli}(\tfrac{1}{2})$, $F_{\theta}={\rm Bernoulli}(0)$ (that is $\theta=-\infty$) and 
$K_{1n} \sim \max({\rm Poisson}(\lambda_n),1)$, with $\lambda_n = a_0 \log n$.
Hence $J(a) = a \log(\tfrac{a}{a_0})-a+a_0$ for $a \geq a_0$ and $J'(a) = \log(\tfrac{a}{a_0})$.

Since $\psi(\theta) = -\log 2$ and $I(\nu) < \infty$ only for $0 \leq \nu \leq 1$,
with $I(0)=\log 2$,
it follows from (\ref{gH})--(\ref{bH}) that for $a_0 \leq \tfrac{1}{\log 2}$,
\begin{eqnarray} \label{bH2}
b_H(\theta) & = & \max_{a \leq \frac{1}{\log 2}} f_H(a), \\ \label{fH2}
\mbox{where } f_H(a) & = & a \log 2 + \tfrac{1}{2}[1-a \log 2-2 J(a)] \\ \nonumber
& = & \tfrac{1}{2}(1+a \log 2)-J(a).
\end{eqnarray}
Note that $\theta \nu=0$ for $\nu=0$ and $\theta \nu=-\infty$ for $\nu>0$,
hence the maximization of $f_H(\nu,a)$ in (\ref{bH}) occurs at $\nu=0$.
For this reason,
we omitted $\nu$ in (\ref{bH2}) and (\ref{fH2}).

By (\ref{aH2}),
$$a_{H \theta} = a_0 \exp(\tfrac{\psi(2 \theta)}{2}-\psi(\theta))=a_0 \exp(\tfrac{\log 2}{2}) = a_0 \sqrt{2},
$$
so for $a_0 \leq \tfrac{1}{\sqrt{2} \log 2}$ (corresponding to $a_{H \theta} I(0) \leq 1$),
\begin{equation} \label{2.5a}
b_H(\theta) = f_H(a_{H \theta}) = \tfrac{1}{2}+a_0(\sqrt{2}-1).
\end{equation}
For $\tfrac{1}{\sqrt{2} \log 2} < a_0 \leq \tfrac{1}{\log 2}$,
\begin{eqnarray} \label{2.5b}
b_H(\theta) & = & f_H(\tfrac{1}{\log 2}) = 1-J(\tfrac{1}{\log 2}) = 1+\tfrac{\log (a_0 e \log 2)}{\log 2}-a_0 \\ \nonumber
\Rightarrow 2^{b_H(\theta)} & = & (a_0 \log 2)(e^{-a_0 \log 2}) 2e.
\end{eqnarray}

From (\ref{2.5a}),
${\rm HC}_n$ is asymptotically powerful when $2a_0 > \rho_H(\beta)$ 
and asymptotically powerless when $2a < \rho_H(\beta)$, where
$$\rho_H(\beta) =  2(1+\sqrt{2})(\beta-\tfrac{1}{2}) \mbox{ if } \tfrac{1}{2} < \beta < \tfrac{1}{2} +
\tfrac{\sqrt{2}-1}{\sqrt{2} \log 2}.
$$
From (\ref{2.5b}),
the above statement holds for 
$$\rho_H(\beta) = -\tfrac{2 W_{-1}(-\frac{2^{\beta}}{2e})}{\log 2}  \mbox{ if }
\tfrac{1}{2} + \tfrac{\sqrt{2}-1}{\sqrt{2} \log 2} < \beta < 1,
$$
where $W_{-1}(x)$ is the negative solution $y$ of $x=ye^y$.
In short,
the boundary $b_H(\theta)$ of the HC test statistic is consistent with the phase transition curve $\rho_H(\beta)$ 
displayed in equation (13) of Donoho and Kipnis (2022).

\section{Bonferroni and rank-adjustment tests}

For the classical sparse Gaussian mixtures,
the Bonferroni test statistic is known to have a phase transition that is optimal for $\beta \geq \tfrac{3}{4}$
but not for $\beta < \tfrac{3}{4}$.
Likewise when we extend to non-Gaussian mixtures of the from (\ref{mixture}),
the Bonferroni test is optimal for $\theta > \tfrac{1}{2} \theta_a^+$ and $\theta < \tfrac{1}{2} \theta_a^-$,
see (\ref{btheta}).
However the Bonferroni test is not optimal,
even for large $|\theta|$,
when we extend to the heterogeneous settings described in Section 4.

Just as thresholding of the HC test statistic leads to optimality,
we show here that thresholding the Bonferroni test statistic leads to optimality for large $|\theta|$ in heterogeneous settings.

Consider p-values $p_i \sim_{i.i.d.}$ Uniform(0,1) for $1 \leq i \leq n$ and let $p_{(1)}$ be the smallest p-value.
It is known that
\begin{equation} \label{Exp}
n p_{(1)} \Rightarrow {\rm Exp}(1) \mbox{ as } n \rightarrow \infty,
\end{equation}
where Exp(1) is the exponential distribution with mean 1.

Let $p_{(1)k}$ be the smallest p-value among $\{ p_i: i \in A_k \}$,
where $A_k = \{ i: K_{in} \geq k \}$,
and let $n_k = \# A_k$.
In view of (\ref{Exp}),
a natural thresholding of the Bonferroni test statistic gives us the rank-adjustment test statistic
\begin{equation} \label{Rn}
R_n = \min_{k \geq 1} n_k p_{(1)k} = \min_{1 \leq i \leq n} r_i p_i,
\end{equation}
where $r_i = \# \{ j: K_{jn} \geq K_{in} \}$ is the rank of $Y_{in}$ in terms of its sample size $K_{in}$,
with larger sample size corresponding to a smaller rank.
Since
$$\sum_{i=1}^n \tfrac{1}{r_i} \leq \sum_{i=1}^n \tfrac{1}{i} \leq 1+\log n,
$$
it follows from a Bonferroni argument that for possibly dependent p-values $p_i \sim$ Uniform(0,1),
$1 \leq i \leq n$,
$$P(R_n \leq \tfrac{\alpha}{1+\log n}) \leq \sum_{i=1}^n P(p_i \leq \tfrac{\alpha}{r_i(1+\log n)})
\leq \alpha.
$$

There is a worst-case $1+\log n$ multiplicative cost when applying the rank-adjustment test statistic
compared to the usual Bonferroni test.
Since $\log n$ grows slowly with $n$,
so for large $n$,
if there is exactly one false null p-value,
the cost is relatively small if its rank is large but the gain can be substantial if its rank is small.

\subsection{Phase transitions of the Bonferroni and rank-adjustment tests}

The phase transition curve of the Bonferroni test,
which is based on the smallest p-value $p_{(1)}$,
shares part of the phase transition curve of ${\rm HC}_n$.
For $\theta$ such that $a_0 I(\mu(\theta)) \leq 1$,
let
$$b_B(\theta) = \max_{(\nu,a): a I(\nu)=1} f_{H \theta}(\nu,a),
$$
where $f_{H \theta}$ is defined in (\ref{gH}) and (\ref{fH}).
It has the representation
$$b_B(\theta) = f_{H \theta}(\mu(\theta^*_H),a_{H \theta}^*),
$$
where $(\theta^*_H,a_{H \theta}^*)$ is characterized by (\ref{J1a}), (\ref{astar}) and $\tfrac{\theta_H^*}{\theta} \geq 1$.

Likewise the phase transition curve of the rank-adjustment test shares part of the phase transition curve of ${\rm HC}_n^{\rm thres}$.
For $\theta$ such that $a_0 I(\mu(\theta)) \leq 1$, 
let
$$b_R(\theta) = \max_{(\nu,a): g(\nu,a)=1} f_{\theta}(\nu,a),
$$
where $f_{\theta}$ is defined in (\ref{gnu}) and (\ref{fnu}).
It has the representation
$$b_R(\theta) =  f_{\theta}(\mu(\theta^*),a_{\theta}^*),
$$
where $(\theta^*,a_{\theta}^*)$ is characterized by (\ref{J2a}),
(\ref{J3}) and $\tfrac{\theta^*}{\theta} \geq 1$.

\begin{thm} \label{thm4} 
Consider the sparse mixture problem in Theorem {\rm 2}. 
The Bonferroni test is asymptotically powerless when 
$\beta > b_B(\theta)$ and asymptotically powerful when $\beta<b_B(\theta)$. 
The rank-adjustment test is asymptotically powerless when $\beta > b_R(\theta)$ 
and asymptotically powerful when $\beta < b_R(\theta)$.  
\end{thm}

\section{Numerical studies}

We plot in Figure 1 the detection boundaries of HC$_n^{\rm thres}$, HC$_n$,
Bonferroni test and rank-adjustment test when $F_0$ is standard normal and $K_{1n} \sim \max(1,{\rm Poisson}(\lambda_0))$
for $\lambda_0  =a_0 \log n$.

There is a roughly constant
gap between the boundaries of the rank-adjustment and Bonferroni tests for the curves displayed in Figure 1.
The gap is smaller between the boundaries of HC$_n^{\rm thres}$ and HC$_n$ especially when $\theta$ is small.
This agrees with the numerical simulation plots in Figure 2 which shows the improvement of HC$_n^{\rm thres}$
over HC$_n$ to be smaller than that of the rank-adjustment test over the Bonferroni test.

The Poisson distribution has a small variance-to-mean ratio of 1,
so the variation of the sample sizes $K_{in}$ is small when it is Poisson distributed. 
To simulate sample size distributions with larger variance-to-mean ratios,
we considered a negative binomial distribution of $K_{1n}$,
with success probability parameter $p=\tfrac{1}{a_0 \log n}$.
At $n=10^5$ and $a_0=0.5$, 
the variance-to-mean ratio is $p^{-1} = 5.76$.
As expected, see Figure 2,
the improvement of HC$_n^{\rm thres}$ over HC$_n$ is larger when $K_{1n}$ has a negative binomial distribution.

For completeness we also compared against the chi-squared test statistic
$\chi^2_n = \sum_{i=1}^n \tfrac{Y_i^2}{K_{in}}$.
Its detection powers are much smaller for sparse mixtures.

The details behind the simulation plots are as follows.
We first generated $n=10^5$ values of $K_{in}$ for both the Poisson and negative binomial distributions.
The values of the test statistics under the
null $\epsilon_n=0$ for 999 independent runs are computed, 
and the 50th largest or smallest ranked value is taken to be the critical value corresponding to a Type I error probability
of 0.05. 
The detection power of a test statistic is the fraction of times it exceeds the computed critical value over 1000 runs,
for a sparse mixture with $\epsilon_n=n^{-\beta}$.

For the threshold HC test statistic, 
to reduce computational cost,
we maximized HC$_n$ over only four values of $k$,
corresponding to $\max \{ k: n_k \geq m \}$ for $m=3 \times 10^3$, $10^4$, $3 \times 10^4$ and $10^5$.
Likewise in the computation of $R_n$, 
we minimized $n_k p_{(1)k}$ over these four values of $k$.

\begin{figure}
\includegraphics[width=5in]{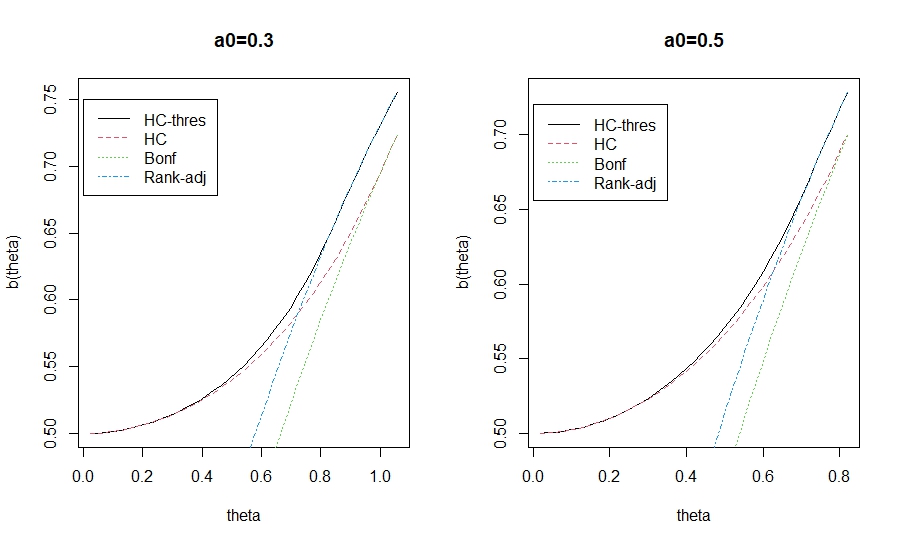}
\caption{ The boundaries $b_J(\theta)$, $b_H(\theta)$, $b_B(\theta)$ and $b_R(\theta)$
for the ${\rm HC}_n^{\rm thres}$, ${\rm HC}_n$, Bonferroni and rank-adjustment test statistics.}
\end{figure}

\begin{figure}
\includegraphics[width=5in]{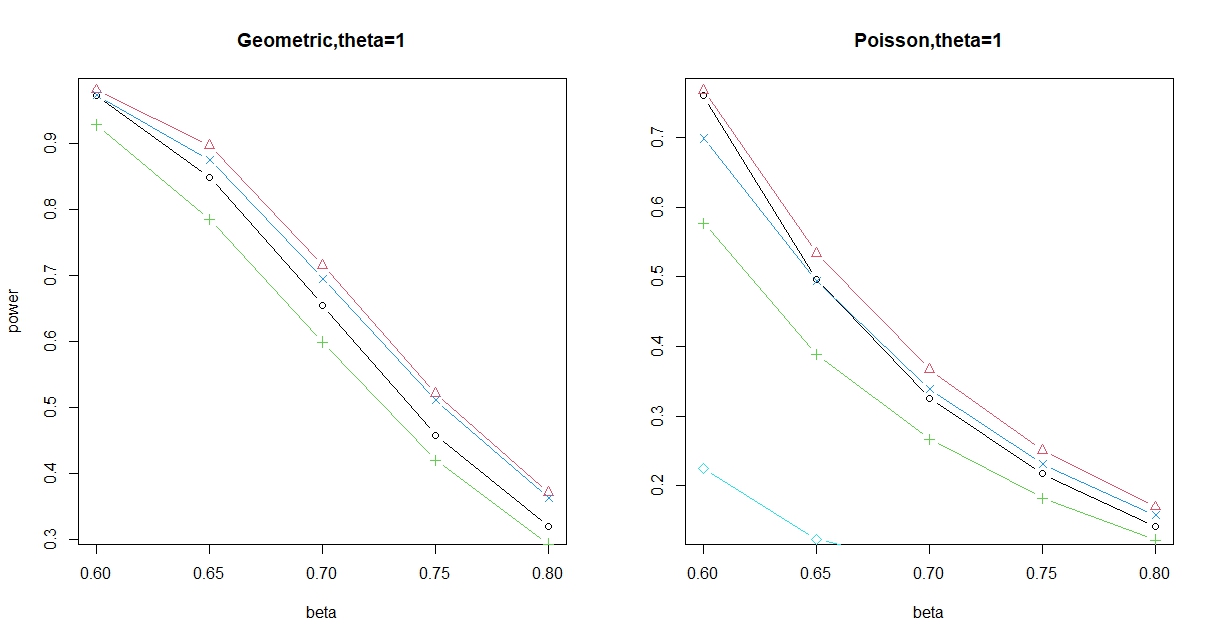}
\includegraphics[width=5in]{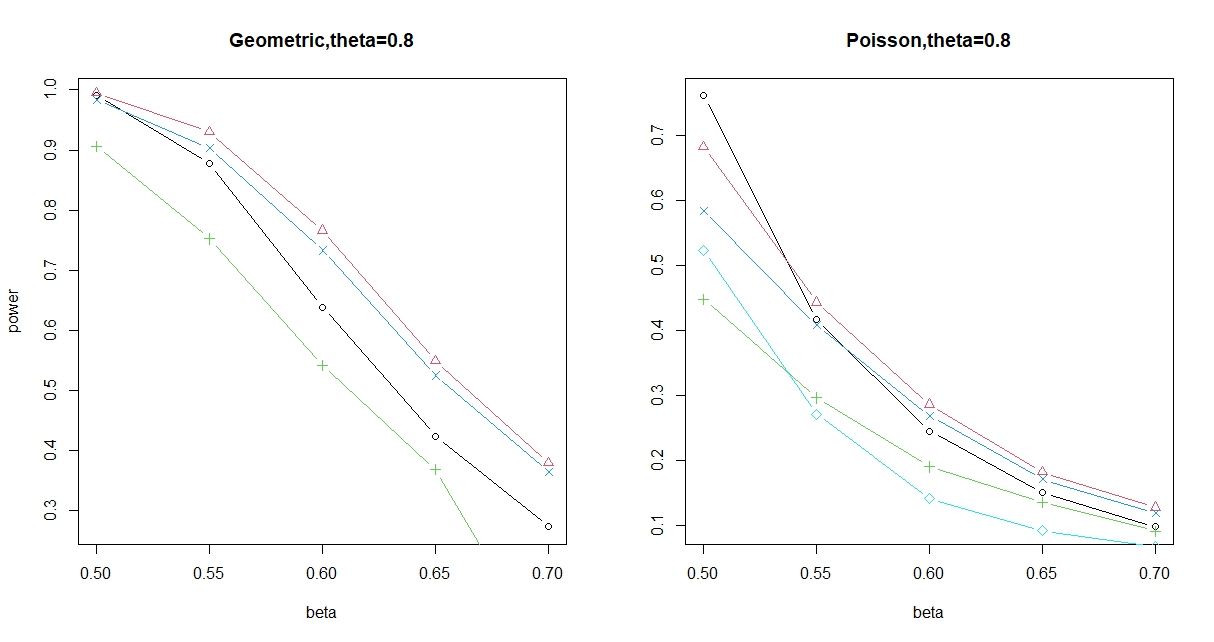}
\caption{ The detection powers of HC$_n$ (black dot), HC$_n^{\rm thres}$ (red triangle),
Bonferroni test (green plus), rank-adjustment test (blue cross) and chi-squared test (light blue diamond),
corresponding to $a_0=0$, $n=10^5$ and Type I error probability of 0.05.}
\end{figure}

\clearpage
\setcounter{equation}{0}
\setcounter{thm}{0}
\setcounter{section}{0}
\setcounter{page}{1}

\begin{center}
{\bf SUPPLEMENT TO ``THRESHOLDING THE HIGHER CRITICISM TEST STATISTIC FOR OPTIMALITY IN A HETEROGENEOUS SETTING''}

\medskip 
Hock Peng Chan\\
{\it Department of Statistics and Data Science} \\
{\it National University of Singapore} 

\end{center}

\section{Introduction}
We provide in this supplement the proofs of Theorems \ref{thm1}--\ref{thm4} of Chan (2023).

\subsection{Notations} 

For two sequences $(a_n)$ and $(b_n)$:
$a_n=O(b_n)$ means that $\tfrac{a_n}{b_n}$ is bounded, 
$a_n=o(b_n)$ means that $\tfrac{a_n}{b_n} \rightarrow 0$,
and $a_n \sim b_n$ means that $\tfrac{a_n}{b_n} \rightarrow 1$.
When $(a_n)$ and $(b_n)$ are random: 
$a_n=O_p(b_n)$ means that $\tfrac{a_n}{b_n}$ is bounded in probability,
$a_n=o_p(b_n)$ means that $\tfrac{a_n}{b_n} \rightarrow 0$ in probability,
and $\stackrel{p}{\rightarrow}$ means converges in probability. 
Let $P_0$, $E_0$, ${\rm Var}_0$ and $P_1$, $E_1$, ${\rm Var}_1$ denote probability, expectation and variance,
under the global null and alternative respectively.
Let $P_{\theta}$, $E_{\theta}$, ${\rm Var}_{\theta}$ denote probability, expectation and variance
with respect to $X_{ij} \sim_{\rm i.i.d.} F_{\theta}$.
Let $\#A$ denote the number of elements in a set $A$,
and let ${\bf 1}$ be the indicator function.
Let $\lceil \cdot \rceil$ denote the greatest integer function.
Let $X_n \Rightarrow F$ denote $X_n$ converging in distribution to $F$.

We consider an exponential family $\{ F_{\theta} \}_{\theta \in \Theta}$ satisfying
$$dF_{\theta}(x) = e^{\theta x-\psi(\theta)} dF_0(x),
$$
with $e^{\psi(\theta)} = E_0 e^{\theta X}$ and $\Theta = \{ \theta: \psi(\theta) < \infty \}$.
The distribution $F_{\theta}$ has mean $\mu(\theta)=\psi'(\theta)$ and the rate function $I$ of $F_0$
satisfies $I(\mu(\theta)) = \theta \mu(\theta)-\psi(\theta)$.

\section{Upper bounds of HC$_n$ and HC$_n^{\rm thres}$ under the global null}

Let $r_i = \# \{ j: K_{jn} \geq K_{in} \}$ be the rank of the sample size $K_{in}$,
for $1 \leq i \leq n$.
Assume without loss of generality,
by re-arranging the hypotheses if necessary,
that $r_1 \leq \cdots \leq r_n$.
Let $p_i$ be the p-value of the $i$th hypotheses.
Let $G_n(t)$ and $U_n(t)$ denote the empirical distribution function and
empirical process of the p-values.
That is,
$$G_n(t) = \tfrac{1}{n} \sum_{i=1}^n {\bf 1}_{\{ p_i \leq t \}} \mbox{ and }
U_n(t) = \sqrt{n}[G_n(t)-t] \mbox{ for } 0 \leq t \leq 1.
$$
Following the notations in Shorack and Wellner (2009), 
let 
$$Z_n(t) = \tfrac{U_n(t)}{\sqrt{t(1-t)}} \mbox{ and } \| Z_n^+ \|_0^p = \sup_{0 < t \leq p} Z_n(t).
$$
It follows that ${\rm HC}_n \leq  \| Z_n^+ \|_0^{\frac{1}{2}}$ and  
${\rm HC}^{\rm thres}_n \leq \max_{m: 2 \leq m \leq n} \| Z_m^+ \|_0^{\frac{1}{2}}$.

\begin{lem} \label{lem1}
Under the global null $X_{ij} \sim_{i.i.d.} F_0$,
$$P_0({\rm HC}_n \geq \log n) \leq P_0({\rm HC}^{\rm thres}_n \geq \log n) \rightarrow 0.
$$
\end{lem}

\begin{proof}
It follows from Cs\'{a}ki (1977), see also Chapter 16 Theorem 3 of Shorack and Wellner (2009),
that for increasing $c_n$, 
\begin{equation} \label{csaki}
\tfrac{\| Z_n^+ \|_0^{\frac{1}{2}}}{\sqrt{c_n}} \rightarrow 0 \mbox{ a.s. }
\mbox{ if } \sum_{n=1}^{\infty} \tfrac{1}{nc_n} < \infty.
\end{equation}
Lemma \ref{lem1} follows from (\ref{csaki}) with $c_n = (\log n)^2$.
\end{proof}

\section{Proof of Theorem \ref{thm1}}

By considering $-X_{ij}$ instead of $X_{ij}$ when $\theta < 0$,
in the proofs of Theorems \ref{thm1}--\ref{thm4},
we may assume without loss of generality that $\theta>0$.

Let $Y_i = X_{i1} + \cdots + X_{ik_n}$ and for $\theta$ such that $a I(\mu(\theta)) \leq 1$,
let
$$b(\theta) = \left\{ \begin{array}{ll} \tfrac{1}{2}(1+a[\psi(2 \theta)-2 \psi(\theta)]) & \mbox{ if }
0 < \theta \leq \tfrac{\theta_a^+}{2}, \cr
a[\theta \mu_a^+-\psi(\theta)] & \mbox{ if } \tfrac{\theta_a^+}{2} < \theta \leq \theta_a^+, \end{array} \right.
$$
where $\theta_a^+ > 0$ is such that $\mu_a^+ = \psi'(\theta_a^+)$ satisfy
$$I(\mu_a^+) = a^{-1}.
$$
Note that $\tfrac{k_n}{\log n} \rightarrow a$ implies 
$$e^{k_n} = e^{[a+o(1)] \log n} = n^{a+o(1)}.
$$

\begin{thm} \label{thm1}
Consider the sparse testing problem
$$Y_{in} \sim_{i.i.d.} (1-\epsilon_n) F_0^{k_n} + \epsilon_n F_{\theta}^{k_n}
$$
with $\epsilon_n=0$ under $H_0$ and $\epsilon_n=n^{-\beta}$ for some $0 < \beta < 1$ under $H_1$.
Let $\tfrac{k_n}{\log n} \rightarrow a$ for some $a>0$ as $n \rightarrow \infty$.
If $\beta > b(\theta)$ then all test statistics are asymptotically powerless.
If $\beta < b(\theta)$ then the ${\rm HC}$ test statistic is asymptotically powerful.
\end{thm}

Let $\delta(\theta) = |b(\theta)-\beta|$.

\subsection{Asymptotically powerless}

We show here that all test statistics are asymptotically powerless when $b(\theta) < \beta$.
Let the risk of a test statistic $T$ be defined by
$$\mbox{Risk}(T) = \inf_{c > 0} [P_0(T \geq c)+P_1(T<c)].
$$
The likelihood ratio test (LRT) statistic 
$$L_n = \prod_{i=1}^n L_{in}, \mbox{ where } L_{in} = 1-n^{-\beta} + n^{-\beta} e^{\theta Y_i-k_n \psi(\theta)},
$$
has the smallest risk over all test statistics, 
achieved by concluding $H_0$: $\epsilon_n=0$ when $L_n < 1$
and $H_1$: $\epsilon_n = n^{-\beta}$ when $L_n \geq 1$.

\begin{lem} \label{lem2} The risk of $L_n$ goes to $1$ if
\begin{equation} \label{Pc}
P_1(L_n > c) \rightarrow 0 \mbox{ for all } c > 1.
\end{equation}
\end{lem}

\begin{proof}
By (\ref{Pc}),
\begin{eqnarray*}
2-{\rm Risk}(L_n) & = & P_0(L_n < 1) + P_1(1 \leq L_n \leq c) + P_1(L_n>c) \cr
& \leq & P_0(L_n < 1) + c P_0(1 \leq L_n \leq c) + P_1(L_n>c) \cr
& \leq & c + P_1(L_n > c) \rightarrow c.
\end{eqnarray*}
Hence $\liminf {\rm Risk}(L_n) \geq 2-c$, and Lemma \ref{lem2} holds by selecting $c > 1$ arbitrarily close to 1.
\end{proof}

Consider $0 < \theta \leq \theta_a^+$ and let
\begin{equation} \label{muastar}
\mu_a^* = \mu_a^+ + \tfrac{\delta(\theta)}{2a \theta}.
\end{equation}
Let $A_{\theta} = \{ \max_i Y_{in} > k_n \mu_a^* \}$ and $\ell_n = \sum_{i=1}^n \ell_{in}$,
where
\begin{equation} \label{lni}
\ell_{in} = \left\{ \begin{array}{ll} \log L_{in} & \mbox{ if } Y_{in} \leq k_n \mu_a^*, \cr
0 & \mbox{ otherwise.} \end{array} \right.
\end{equation}
Lemma \ref{lem3} is proved in Section 3.3.

\begin{lem} \label{lem3} 
If $b(\theta) < \beta$,
then as $n \rightarrow \infty$,

\smallskip \noindent (a) $P_1(A_{\theta}) \rightarrow 0$.

\smallskip \noindent (b) $\limsup E_1 \ell_n \leq 0$.

\smallskip \noindent (c) ${\rm Var}_1 \ell_n \rightarrow 0$.
\end{lem}

\begin{proof}[Proof of Theorem 1 for $b(\theta) < \beta$]
By Lemma \ref{lem3}, Chebyshev's inequality and (\ref{lni}),
for $c>1$,
$$P_1(L_n > c) \leq P_1(A_{\theta})+\tfrac{{\rm Var}_1 \ell_n}{(\log c-E_1 \ell_n)^2} \rightarrow 0.
$$
By Lemma \ref{lem2},
risk of $L_n$ goes to 1,
hence all test statistics are asymptotically powerless.
\end{proof}

\subsection{Asymptotically powerful}

We show here that ${\rm HC}_n$ is asymptotically powerful when $b(\theta) > \beta$.
By Lemma \ref{lem1}, 
it suffices to show $P_1({\rm HC}_n \geq \log n) \rightarrow 1$.

\begin{proof}[Proof of Theorem 1 for $b(\theta) > \beta$]
For $0 < \theta \leq \tfrac{\theta_a^+}{2}$, 
$$\delta(\theta) = \tfrac{1}{2} \{ 1+a[\psi(2 \theta)-2 \psi(\theta)] \} -\beta.
$$
By Cram\'{e}r's Theorem,
\begin{eqnarray} \label{pn}
p_n & := & 2 P_0(Y_1 \geq k_n \mu(2 \theta)) \\ \nonumber
& = & e^{k_n[I(\mu(2 \theta))+o(1)]} \\ \nonumber
& = & n^{-a I(\mu(2 \theta))+o(1)}, \\ \label{qn}
q_n & := & P_{\theta}(Y_1 \geq k_n \mu(2 \theta)) \\ \nonumber
& = & e^{-k_n[I(\mu(2 \theta))-\theta \mu(2 \theta)+\psi(\theta)+o(1)]} \\ \nonumber
& = & n^{-a[I(\mu(2 \theta))-\theta \mu(2 \theta)+\psi(\theta)]+o(1)}.
\end{eqnarray}
By Chebyshev's inequality, under $P_1$,
$$\# \{ i: p_{(i)} \leq p_n \} = np_n+O_p(\sqrt{np_n}) + [1+o_p(1)] n^{1-\beta} q_n,
$$
provided $n^{1-\beta} q_n \rightarrow \infty$.
Hence 
\begin{equation} \label{HCn}
{\rm HC}_n \geq \tfrac{\# \{ i: p_{(i)} \leq p_n \}-n p_n}{\sqrt{np_n}} 
= [1+o_p(1)] \sqrt{\tfrac{n^{1-2 \beta} q_n^2}{p_n}} + O_p(1).
\end{equation}
By (\ref{pn}) and (\ref{qn}),
\begin{eqnarray*} 
\sqrt{\tfrac{n^{1-2 \beta} q_n^2}{p_n}} & = & n^{\frac{1}{2}-\beta-a[\frac{1}{2} I(\mu(2 \theta))-\theta \mu(2 \theta)
+\psi(\theta)]+o(1)} \\ \nonumber
& = & n^{\frac{1}{2}-\beta+a[\psi(2 \theta)/2-\psi(\theta)]+o(1)} \\ \nonumber
& = & n^{\delta(\theta)+o(1)}, \\ 
n^{1-\beta} q_n & = & n^{1-\beta-a[\psi(2 \theta)/2-\psi(\theta)]-\frac{1}{2} a I(\mu(2 \theta))+o(1)} \\ \nonumber
& = & n^{\delta(\theta)+\frac{1}{2}[1-a I(\mu(2 \theta))]+o(1)}.
\end{eqnarray*}
Since $a I(\mu(2 \theta)) \leq 1$,
we conclude $P_1({\rm HC}_n \geq \log n) \rightarrow 1$ from (\ref{HCn}).

For $\tfrac{\theta_a^+}{2} < \theta \leq \theta_a^+$,
$$\delta(\theta) = a[\theta \mu_a^+-\psi(\theta)]-\beta.
$$
By Cram\'{e}r's Theorem,
\begin{eqnarray} \label{pn2}
p_n & := & 2 P_0(Y_1 \geq k_n \mu_a^+) \\ \nonumber
& = & e^{-k_n[I(\mu_a^+)+o(1)]} \\ \nonumber
& = & n^{-1+o(1)}, \\ \label{qn2}
q_n & := & P_{\theta}(Y_1 \geq k_n \mu_a^+) \\ \nonumber
& = & e^{-k_n[I(\mu_a^+)-\theta \mu_a^++\psi(\theta)+o(1)]} \\ \nonumber
& = & n^{-1+a[\theta \mu_a^+-\psi(\theta)]+o(1)}.
\end{eqnarray}
By (\ref{pn2}) and (\ref{qn2}),
\begin{eqnarray*}
\sqrt{\tfrac{n^{1-2 \beta} q_n^2}{p_n}} & = & n^{-\beta+a[\theta \mu_a^+-\psi(\theta)]+o(1)} = n^{\delta(\theta)+o(1)}, \cr
n^{1-\beta} q_n & = & n^{a[\theta \mu_a^+-\psi(\theta)]-\beta+o(1)} \cr
& = & n^{\delta(\theta)+o(1)} \rightarrow \infty,
\end{eqnarray*}
and $P_1({\rm HC}_n \geq \log n) \rightarrow 1$ follows from (\ref{HCn}). 
\end{proof}

\subsection{Proof of Lemma \ref{lem3}}

\begin{proof}[Proof of Lemma \ref{lem3}(a)]
Since $\mu_a^* > \mu_a^+$,
so $a I(\mu_a^*) > a I(\mu_a^+) = 1$.
By Cram\'{e}r's Theorem,
\begin{eqnarray*}
P_1(A_{\theta}) & = & (1-n^{-\beta}) P_0(A_{\theta}) + n^{-\beta} P_{\theta}(A_{\theta}) \cr
& \leq & n P_0(Y_1 > k_n \mu_a^*) + n^{1-\beta} P_{\theta}(Y_1 > k_n \mu_a^+) \cr
& \leq & n e^{-k_n I(\mu_a^*)} + n^{1-\beta} e^{-k_n[I(\mu_a^+)-\theta \mu_a^++\psi(\theta)]} \cr
& = & n^{1-a I(\mu_a^*)+o(1)} + n^{1-\beta-a I(\mu_a^+)+a[\theta \mu_a^+ - \psi(\theta)]+o(1)},
\end{eqnarray*}
and Lemma \ref{lem3}(a) follows from Lemma \ref{lem4} below.
\end{proof}

\begin{lem} \label{lem4}
Consider $0 < \theta \leq \theta_a^+$. 
If $\beta > b(\theta)$ then $\beta > a[\theta \mu_a^+ - \psi(\theta)]$.
\end{lem}

\begin{proof}
Let $g(\theta) = \tfrac{1}{2}(1+a[\psi(2 \theta)-2 \psi(\theta)])$ and $h(\theta)=a[\theta \mu_a^+-\psi(\theta)]$.
It suffices to show $g(\theta) \geq h(\theta)$ for $0 < \theta \leq \tfrac{\theta_a^+}{2}$.
Check that
\begin{eqnarray} \label{hg}
h(\tfrac{\theta_a^+}{2}) & = & a[\tfrac{\theta_a^+ \mu_a^+}{2}-\psi(\tfrac{\theta_a^+}{2})] \\ \nonumber
& = & a[\tfrac{I(\mu_a^+)+\psi(\theta_a^+)}{2} - \psi(\tfrac{\theta_a^+}{2})] \\ \nonumber
& = & g(\tfrac{\theta_a^+}{2}).
\end{eqnarray}
For $0 < \theta < \tfrac{\theta_a^+}{2}$,
\begin{eqnarray*}
g'(\theta) & = & a[\psi'(2 \theta)-\psi'(\theta)] =  a[\mu(2 \theta)-\mu(\theta)], \cr
h'(\theta) & = & a[\mu_a^+-\psi(\theta)] = a[\mu(\theta_a^+)-\mu(\theta)].
\end{eqnarray*}
Since $\mu(\cdot)$ is increasing,
we conclude $g'(\theta) < h'(\theta)$, and Lemma \ref{lem4} follows from (\ref{hg}).
\end{proof}

\begin{proof}[Proof of Lemma \ref{lem3}(b)]
Since $\log(1+x) \leq x$ and $E_0 e^{\theta Y_1-k_n \psi(\theta)} = 1$,
for $0 < \theta \leq \tfrac{\theta_a^+}{2}$,
$$2 \beta = 1+a[\psi(2 \theta)-2 \psi(\theta)]+2 \delta(\theta).
$$
Hence
\begin{eqnarray} \label{E1}
E_1 \ell_{1n} & \leq & n^{-2 \beta} (E_{\theta} e^{\theta Y_1 - k_n \psi(\theta)}-1) \\ \nonumber
& \leq & n^{-2 \beta} e^{k_n[\psi(2 \theta)-2 \psi(\theta)]} \\ \nonumber
& = & n^{-2 \beta+a[\psi(2 \theta)-2 \psi(\theta)]+o(1)} \\ \nonumber
& = & n^{-2 \delta(\theta)-1+o(1)},
\end{eqnarray}
and so $\limsup E_1 \ell_n = \limsup (n E_1 \ell_{1n}) \leq 0$.

For $\tfrac{\theta_a^+}{2} < \theta \leq \theta_a^+$,
by (\ref{muastar}),
\begin{eqnarray*}
2 a \theta \mu_a^* & = & 2a \theta \mu_a^+ + \delta(\theta), \cr
2 \beta &  =& 2a[\theta \mu_a^+ - \psi(\theta)] + 2 \delta(\theta).
\end{eqnarray*}
Hence
\begin{eqnarray} \label{E2}
E_1 \ell_{1n} & \leq & n^{-2 \beta} E_{\theta}(e^{\theta Y_1-k_n \psi(\theta)} {\bf 1}_{\{ Y_1 \leq k_n \mu_a^* \}}) \\ \nonumber
& \leq & n^{-2 \beta} E_{\theta_a^+}(e^{(2 \theta-\theta_a^+) Y_1-k_n[2 \psi(\theta)-\psi(\theta_a^+)]}
{\bf 1}_{\{ Y_1 \leq k_n \mu_a^* \}}) \\ \nonumber
& \leq & n^{-2 \beta+a[(2 \theta-\theta_a^+) \mu_a^+-2 \psi(\theta)+\psi(\theta_a)]+\delta(\theta)+o(1)} \\ \nonumber
& = & n^{-\delta(\theta)-a I(\mu_a^+)+o(1)} \\ \nonumber
& = & n^{-\delta(\theta)-1+o(1)},
\end{eqnarray}
and so $\limsup E_1 \ell_n = \limsup (n E_1 \ell_{1n}) \leq 0$.
\end{proof}

\begin{proof}[Proof of Lemma \ref{lem3}(c)]
Consider first $0 < \theta \leq \tfrac{\theta_a^+}{3}$.
Since $|\log(1+x)|^2 \leq 4x^2$ for $x \geq -\tfrac{1}{2}$,
\begin{eqnarray} \label{V1}
{\rm Var}_1 \ell_n & = & n {\rm Var}_1 \ell_{1n} \\ \nonumber
& \leq & n E_1 \ell_{1n}^2 \\ \nonumber
& \leq & 4n(1-n^{-\beta}) n^{-2 \beta} E_0 (e^{\theta Y_1-k_n \psi(\theta)}-1)^2 \\ \nonumber
& & \qquad + 4n^{1-3 \beta} E_{\theta}(e^{\theta Y_1-k_n \psi(\theta)}-1)^2.
\end{eqnarray}
Check that
\begin{eqnarray} \label{V2}
E_0(e^{\theta Y_1-k_n \psi(\theta)}-1)^2 & \leq & E_0 e^{2 \theta Y_1-2k_n \psi(\theta)} \\ \nonumber
& = & e^{k_n[\psi(2 \theta)-2\psi(\theta)]} \\ \nonumber
& = & n^{a[\psi(2 \theta)-2 \psi(\theta)]+o(1)} \\ \nonumber
& = & n^{2 \beta-2 \delta(\theta)-1+o(1)}, \\ \label{V3}
E_{\theta} (e^{\theta Y_1-k_n \psi(\theta)}-1)^2 & \leq & E_0 e^{3 \theta Y_1-3 k_n \psi(\theta)} \\ \nonumber
& = & e^{k_n[\psi(3 \theta)-3 \psi(\theta)]} \\ \nonumber
& = & n^{a[\psi(3 \theta)-3 \psi(\theta)]+o(1)}.
\end{eqnarray}
We conclude ${\rm Var}_1 \ell_n \rightarrow 0$ from (\ref{V1})--(\ref{V3}) and Lemma \ref{lem6} below.

For $\tfrac{\theta_a^+}{3} < \theta \leq \theta_a^+$,
\begin{eqnarray} \label{V4}
{\rm Var}_1 \ell_n & \leq & n E_1 \ell_{1n}^2 \\ \nonumber
& \leq & 4n(1-n^{-\beta}) n^{-2 \beta} E_0(e^{\theta Y_1-k_n \psi(\theta)} {\bf 1}_{\{ Y_1 \leq k_n \mu_a^* \}}
-1)^2 \\ \nonumber
& & \qquad + 4n^{1-3 \beta} E_{\theta} (e^{\theta Y_1-k_n \psi(\theta)} {\bf 1}_{\{ Y_1 \leq k_n \mu_a^* \}}-1)^2.
\end{eqnarray}
By the calculations in (\ref{E2}),
\begin{eqnarray} \label{3.19}
& & n^{1-2 \beta} E_0(e^{\theta Y_1-k_n \psi(\theta)} {\bf 1}_{\{ Y_1 \leq k_n \mu_a^* \}}-1)^2 \\ \nonumber
& \leq & n^{1-2 \beta} E_{\theta}(e^{\theta Y_1-k_n \psi(\theta)} {\bf 1}_{\{ Y_1 \leq k_n \mu_a^* \}}) \\ \nonumber
& = & n^{-\delta(\theta)+o(1)} \rightarrow 0.
\end{eqnarray}
Moreover,
\begin{eqnarray} \label{3.20}
& & n^{1-3 \beta} E_{\theta} (e^{\theta Y_1-k_n \psi(\theta)} {\bf 1}_{\{ Y_1 \leq k_n \mu_a^* \}}-1)^2 \\ \nonumber
& \leq & n^{1-3 \beta} E_{\theta_a^+}(e^{(3 \theta-\theta_a^+) Y_1-k_n[3 \psi(\theta)-\psi(\theta_a^+)]}
{\bf 1}_{\{ Y_1 \leq k_n \mu_a^* \}}) \\ \nonumber
& \leq & n^{1-3 \beta+a[(3 \theta-\theta_a^+) \mu_a^+-3 \psi(\theta)+\psi(\theta_a^+)]+3 \delta(\theta)/2+o(1)} \\ \nonumber
& = & n^{1-3 \delta(\theta)/2-a I(\mu_a^+)+o(1)} = n^{-3 \delta(\theta)/2+o(1)} \rightarrow 0,
\end{eqnarray}
We conclude ${\rm Var}_1 \ell_n \rightarrow 0$ from (\ref{V4})--(\ref{3.20}).
\end{proof}

\begin{lem} \label{lem6} 
Consider $0 < \theta \leq \tfrac{\theta_a^+}{3}$.
If $\beta > b(\theta)$ then
$$1-3 \beta +a[\psi(3 \theta)-3 \psi(\theta)] < 0.
$$
\end{lem}

\begin{proof}
Consider $0 < \theta \leq \tfrac{\theta_a^+}{3}$ and let 
\begin{eqnarray*}
g(\theta) & = & \tfrac{1}{2}(1+a[\psi(2 \theta)-2 \psi(\theta)]), \cr
j(\theta) & = & \tfrac{1}{3}(1+a[\psi(3 \theta)-3 \psi(\theta)]).
\end{eqnarray*}
Since $\beta > \tfrac{1}{2}(1+a[\psi(2 \theta)-2 \psi(\theta)])$,
it suffices to show that $g(\theta) \geq j(\theta)$.
Since $\mu(\cdot)=\psi'(\cdot)$ is increasing,
\begin{eqnarray*}
6[g(\theta)-j(\theta)] & = & 1+a[3 \psi(2 \theta)-2 \psi(3 \theta)] \cr
& = & 1 + 3a [\psi(2 \theta)-\psi(3 \theta)] + a \psi(3 \theta) \cr
& \geq & 1-3a \theta \psi'(3 \theta) + a \psi(3 \theta) \cr
& = & 1-a[3 \theta \mu(3 \theta)-\psi(3 \theta)] \cr
& =& 1-a I(\mu(3 \theta)) \cr
& \geq & 0,
\end{eqnarray*}
with the last inequality following from $I(\mu(3 \theta)) \leq I(\mu_a^+)=a^{-1}$.
\end{proof} 

\section{Proof of Theorem \ref{thm2}}

Let $Y_{in} = \sum_{j=1}^{K_{in}} X_{ij}$,
with $X_{ij} \sim_{\rm i.i.d.} F_0$.
Consider a function $J$ such that:

\smallskip \noindent
(A1) $J(a)=0$ for $0 \leq a \leq a_0$ for some $a_0>0$,

\smallskip \noindent
(A2) $J'(a)$ is continuous and increasing on $[a_0,\infty)$ with $J'(a_0)=0$ and $\lim_{a \rightarrow \infty} J'(a)=\infty$,

\smallskip \noindent
(A3) $P(K_{in} = k) = \alpha_{kn} n^{-J(\frac{k}{\log n})}$ with $\sup_{k \geq \lambda_n} (\tfrac{|\log \alpha_{kn}|}{k})
\rightarrow 0$,
for $\lambda_n = a_0 \log n$.

\medskip \noindent
The proof of Theorem \ref{thm2} applies Lemma \ref{pre}, 
which is proved in Section 4.3.

\begin{lem} \label{pre}
Let $a_1$ be such that $J(a_1)=1$.
For any $\omega>0$,
$$\sum_{k \leq a_1 \log n} e^{J(\frac{k}{\log n})} P(K_{1n}=k) + n P(K_{1n} > a_1 \log n) = o(n^{\omega}).
$$
\end{lem}

For $\theta$ such that $a_0 I(\mu(\theta)) \leq 1$,
define 
\begin{eqnarray} \label{bJtheta}
b_J(\theta) & = & \max_{(\nu,a): g(\nu,a) \leq 1} f_{\theta}(\nu,a), \\ \label{gnu}
\mbox{where } g(\nu,a) & = &  a I(\nu)+J(a), \\ \label{fnu}
f_{\theta}(\nu,a) & = & a[\theta \nu-\psi(\theta)] + \tfrac{1}{2}[1-g(\nu,a)].
\end{eqnarray}

\begin{thm} \label{thm2}
Consider the test of $H_0$: $\epsilon_n=0$ versus $H_1$: $\epsilon_n=n^{-\beta}$ for some $0 < \beta < 1$,
with sample sizes $K_{in}$ i.i.d. with rate function $J$ satisfying (A1)--(A3).
If $\beta > b_J(\theta)$ then all test statistics are asymptotically powerless.
If $\beta < b_J(\theta)$ then ${\rm HC}_n^{\rm thres}$ is asymptotically powerful.
\end{thm}

Let $\delta(\theta) = |b_J(\theta)-\beta|$.

\subsection{Asymptotically powerless}

We show here that all test statistics are asymptotically powerless when $b_J(\theta) < \beta$.
The LRT statistic
$$L_n = \prod_{i=1}^n L_{in}, \mbox{ where } L_{in} = 1-n^{-\beta}+n^{-\beta} e^{\theta Y_1-K_{in} \psi(\theta)}.
$$
It has the smallest risk over all test statistics,
achieved by concluding $H_0$: $\epsilon_n=0$ when $L_n < 1$ and $H_1$: $\epsilon=n^{-\beta}$ when $L_n \geq 1$.
By Lemma \ref{lem2},
to show that the risk of the LRT statistic tends to 1,
it suffices to show 
\begin{equation} \label{Lnc}
P_1(L_n > c) \rightarrow 0 \mbox{ for all } c>1.
\end{equation}

Consider $\theta>0$ such that $a_0 I(\mu(\theta)) \leq 1$.
For $a>0$ such that $J(a) \leq 1$,
let
\begin{eqnarray} \label{gc}
\mu_a^+ & = & \sup \{ \nu: g(\nu,a) \leq 1 \}, \\ \label{gb}
\mu_a^* & = & \mu_a^+ + \tfrac{\delta(\theta)}{2a \theta}.
\end{eqnarray}
For $a>0$ such that $J(a)>1$,
let $\mu_a^+ = \mu_a^* = -\infty$.

Let $a_{in} = \tfrac{K_{in}}{\log n}$ and let $\mu_{in}^*$ and $\mu_{in}^+$ be short forms of $\mu^*_{a_{in}}$ and $\mu_{a_{in}}^+$.
Let $\ell_n = \sum_{i=1}^n \ell_{in}$,
where
$$\ell_{in} = \left\{ \begin{array}{ll} \log L_{in} & \mbox{ if } Y_i \leq K_{in} \mu_{in}^*, \cr
0 & \mbox{ otherwise.} \end{array} \right.
$$
Let $A_{\theta} = \{ Y_i > K_{in} \mu_{in}^* \mbox{ for some } i \}$.
Lemma \ref{lem7} is proved in Section 4.3.

\begin{lem} \label{lem7}
If $b_J(\theta) < \beta$ then as $n \rightarrow \infty$,

\smallskip \noindent
(a) $P_1(A_{\theta}) \rightarrow 0$.

\smallskip \noindent
(b) $\limsup E_1 \ell_n \leq 0$.

\smallskip \noindent
(c) ${\rm Var}_1 \ell_n \rightarrow 0$.
\end{lem}

\begin{proof}[Proof of Theorem 2 for $b_J(\theta) < \beta$]
Under $A_{\theta}^c$, 
$\log L_n = \ell_n$.
By Chebyshev's inequality,
for $c>1$,
$$P_1(L_n > c) \leq P_1(A_{\theta})+\tfrac{{\rm Var}_1 \ell_n}{(\log c-E_1 \ell_n)^2}. 
$$
By Lemma \ref{lem7}, 
(\ref{Lnc}) holds.
Hence by Lemma \ref{lem2}, 
risk of $L_n$ goes to 1,
and all test statistics are asymptotically powerless.
\end{proof}

\subsection{Asymptotically powerful}

We show here that ${\rm HC}_n^{\rm thres}$ is asymptotically powerful when $b_J(\theta) > \beta$.
By Lemma \ref{lem1}, 
it suffices to show 
\begin{equation} \label{show2.2} 
P_1({\rm HC}_n^{\rm thres} \geq \log n) = P_1(\max_{1 \leq k \leq n} {\rm HC}_{kn} \geq \log n) \rightarrow 1.
\end{equation}

\begin{proof}[Proof of Theorem 2 for $b_J(\theta) > \beta$]
By (\ref{bJtheta}),
there exists $(\nu,a)$ such that
\begin{eqnarray} \label{bJ}
b_J(\theta) & = & a[\theta \nu-\psi(\theta)]+\tfrac{1}{2}[1-g(\nu,a)], \\ \label{bf}
\mbox{with } g(\nu,a) & = & a I(\nu) + J(a) \leq 1.
\end{eqnarray}
Consider $k = \lceil a \log n \rceil$ and let
\begin{eqnarray*}
p_n^* & = & 2P_0(Y_1 \geq k \nu|K_{1n}=k), \cr
q_n & = & P_{\theta}(Y_1 \geq k \nu|K_{1n}=k).
\end{eqnarray*}
Let $n_k = \# \{ i: K_{in} \geq k \}$ and $n_k^* = \{ i: K_{in}=k \}$.
By (\ref{show2.2}), 
it suffices to show that
\begin{eqnarray} \label{P1HC}
P_1({\rm HC}_{kn}(p_n^*) \geq \log n) & \rightarrow & 1, \\ \label{maxi}
\mbox{where } {\rm HC}_{kn}(p) & = & \tfrac{\# \{ i: K_{in} \geq k, p_i \leq p \}-n_k p}{\sqrt{n_k p(1-p)}}.
\end{eqnarray}

By Chebyshev's inequality,
under $P_1$,
\begin{equation} \label{ch2}
\# \{ i: K_{in} \geq k, p_i \leq p_n^* \} \geq n_k p_n^* +O_p(\sqrt{n_k p_n^*}) + [1+o_p(1)] \epsilon_n n_k^* q_n,
\end{equation}
with $\epsilon_n = n^{-\beta}$,
provided $\epsilon_n n_k^* q_n \stackrel{p}{\rightarrow} \infty$.
By (\ref{maxi}) and (\ref{ch2}),
\begin{eqnarray} \label{HCstar}
{\rm HC}_{kn}(p_n^*) \geq  [1+o_p(1)] \tfrac{\epsilon_n n_k^* q_n}{n_k p_n^*} +O_p(1).
\end{eqnarray}
By (A3) and Lemma \ref{pre},
\begin{eqnarray} \label{n1}
n_k & = & n^{1-J(a)+o_p(1)},  \\ \label{n2}
n_k^* & = & n^{1-J(a)+o_p(1)}, \\ \label{p1}
p_n^* & = & n^{-a I(\nu) + o_p(1)}, \\ \label{q1}
q_n & = & n^{-a[I(\nu)-\theta \nu + \psi(\theta)] +o(1)}.
\end{eqnarray}

By (\ref{bJ}) and (\ref{n1})--(\ref{q1}),
\begin{eqnarray*}
\tfrac{\epsilon_n n_k^* q_n}{n_k p_n^*} & = & n^{\frac{1}{2}-\beta-\frac{1}{2} J(a)-\frac{1}{2} a
[I(\nu)-2 \theta \nu+2 \psi(\theta)]+o_p(1)} \cr
& = & n^{\delta(\theta)+o_p(1)}, \cr
\epsilon_n n_k^* q_n & = & n^{1-\beta-J(a)-a[I(\nu)-\theta \nu+\psi(\theta)]+o_p(1)} \cr
& =& n^{1-\beta-g(a,\nu)+a[\theta \nu-\psi(\theta)]+o_p(1)} \cr
& = & n^{\delta(\theta)+\frac{1}{2}[1-g(a,\nu)]+o_p(1)} \stackrel{p}{\rightarrow} \infty,
\end{eqnarray*}
and (\ref{P1HC}) follows from (\ref{HCstar}).
\end{proof}

\subsection{Proofs of Lemmas \ref{pre} and \ref{lem7}}

\begin{proof}[Proof of Lemma \ref{pre}]
Let $\lambda_n = a_0 \log n$ and $\omega>0$.
Since $J(a)=0$ for $a \leq a_0$,
\begin{equation} \label{p61}
\sum_{k \leq \lambda_n} e^{J(\frac{k}{\log n})} P(K_{1n}=k) = P(K_{1n} \leq \lambda_n) \leq 1.
\end{equation}
By (A3),
\begin{eqnarray} \label{p62}
\sum_{\lambda_n < k \leq a_1 \log n} e^{J(\frac{k}{\log n})} P(K_{1n}=k) 
& = & \sum_{\lambda_n < k \leq a_1 \log n} \alpha_{kn} \\ \nonumber
& \leq & (a_1 \log n) \exp \Big( \lambda_n \sup_{k \geq \lambda_n} \tfrac{|\log \alpha_{kn}|}{k} \Big) \\ \nonumber
& = & o(n^{\omega}).
\end{eqnarray}
Since $J'$ is increasing, $J(a) \geq 1+(a-a_1) J'(a_1)$ for $a \geq a_1$ and
by (A3),
\begin{eqnarray} \label{p63}
n \sum_{k > a_1 \log n} P(K_{1n}=k) & \leq & \sum_{k \geq a_1 \log n} \alpha_{kn} e^{-(\frac{k}{\log n}-a_1) J'(a_1)} \\ \nonumber
& \leq & (1+\tfrac{\log n}{J'(a_1)}) \exp \Big( \lambda_n \sup_{k \geq \lambda_n} \tfrac{|\log \alpha_{kn}|}{k} \Big) \\ \nonumber
& = & o(n^{\omega}).
\end{eqnarray}
Lemma \ref{pre} follows from (\ref{p61})--(\ref{p63}).
\end{proof}

\begin{proof}[Proof of Lemma \ref{lem7}(a)]
Since $\mu_{in}^+ \leq \mu_{in}^*$,
\begin{equation} \label{7a1}
P_1(A_{\theta}) \leq n P_0(Y_1 > K_{1n} \mu_{1n}^*) + n^{1-\beta} P_{\theta}(Y_1 > K_{1n} \mu_{1n}^+).
\end{equation}
By (\ref{gc}), (\ref{gb}) and (\ref{bf}), 
\begin{eqnarray} \label{conclude}
a I(\mu_a^*) & \geq & a I(\mu_a^+) + a \theta (\mu_a^*-\mu_a^+) \\ \nonumber
& \geq & 1-J(a)+\tfrac{\delta(\theta)}{2},
\end{eqnarray}
if $g(\mu_a^+,a)=1$,
and $I(\mu_a^*)=\infty$ if $g(\mu_a^+,a)<1$.
By Lemma \ref{pre}, (\ref{conclude}) and Cram\'{e}r's Theorem,
\begin{eqnarray} \label{7a2}
& & n P_0(Y_1 > K_{1n} \mu_{1n}^*) \\ \nonumber
& = & n E(e^{-K_{1n} I(\mu_{1n}^*)} {\bf 1}_{\{ K_{1n} \leq a_1\log n \}}) + n P(K_{1n} > a_1 \log n) \\ \nonumber
& \leq & n^{-\frac{\delta(\theta)}{2}} \Big[ \sum_{k \leq a_1 \log n} e^{J(\frac{k}{\log n})} P(K_{1n}=k) 
+ n P(K_{1n} > a_1 \log n) \Big] \\ \nonumber
& \rightarrow & 0.
\end{eqnarray}
By (\ref{gc}) and (\ref{bf}),
for $a=\tfrac{k}{\log n}$,
\begin{eqnarray} \label{7a3}
& & n^{1-\beta} P_{\theta}(Y_1 > K_{1n} \mu_{1n}^+) \\ \nonumber
& = & n^{1-\beta} [E(e^{-K_{1n}[I(\mu_{1n}^+)-\theta \mu_{1n}^+ + \psi(\theta)]} {\bf 1}_{\{ K_{1n} \leq a_1 \log n \}})
+ P(K_{1n} > a_1 \log n)] \\ \nonumber
& \leq & n^{1-\beta} [E_{\theta}(e^{-K_{1n} I(\mu_a^+)})+P(K_{1n} > a_1 \log n)] \\ \nonumber
& \leq & n^{1-\beta} \Big[ \sum_{k \leq a_1 \log n} n^{a[\theta \mu_{1n}^+-\psi(\theta)]+J(\frac{k}{\log n})-1}
P(K_{1n}=k) + P(K_{1n} > a_1 \log n) \Big] \\ \nonumber
& \leq & n^{-\beta + b_J(\theta)} \Big[ \sum_{k \leq a_1 \log n} e^{J(\frac{k}{\log n})} P(K_{1n}=k) 
+ n P(K_{1n} > a_1 \log n) \Big] \\ \nonumber
& \rightarrow & 0,
\end{eqnarray}
with $b_J(\theta) \geq a[\theta \mu_a^+-\psi(\theta)]$ following from (\ref{gc}) and (\ref{bJ}).

We conclude Lemma \ref{lem7}(a) from (\ref{7a1}), (\ref{7a2}) and (\ref{7a3}).
\end{proof}

\begin{proof}[Proof of Lemma \ref{lem7}(b)]
Since $E_0(e^{\theta Y_1-k \psi(\theta)}|K_{1n}=k) = 0$ for $k \geq 1$ and $\log(1+x) \leq x$,
\begin{equation} \label{mk0}
E_1 \ell_{1n} \leq n^{-2 \beta} \sum_{k \leq a_1 \log n} m_k P(K_{1n}=k), 
\end{equation}
where $m_k = E_{\theta}(e^{\theta Y_1-k \psi(\theta)} {\bf 1}_{\{ Y_1 \leq k \mu_{1n}^* \}}|K_{1n}=k)$.

For $a \leq a_1$ such that $g(\mu(2 \theta),a) \leq 1$,
\begin{eqnarray} \label{bJ1}
b_J(\theta) & \geq & f_{\theta}(\mu(2 \theta),a) \\ \nonumber
& = & a[\theta \mu(2 \theta)-\psi(\theta)]+\tfrac{1}{2} \{ 1-a[2 \theta \mu(2 \theta)-\psi(2 \theta)] - J(a) \} \\ \nonumber
& = & \tfrac{1}{2} \{ a[\psi(2 \theta)-2 \psi(\theta)]+1-J(a) \}.
\end{eqnarray}
Hence for $a=\tfrac{k}{\log n}$,
\begin{eqnarray} \label{mk1}
m_k & \leq & E_0(e^{2 \theta Y_1 - 2k \psi(\theta)}|K_{1n}=k) \\ \nonumber
& = & n^{a[\psi(2 \theta)-2 \psi(\theta)]} \\ \nonumber
& \leq & n^{2 b_J(\theta,a)-1+J(a)}.
\end{eqnarray}

Let $\theta_a^+$ be such that $\psi'(\theta_a^+) = \mu_a^+$.
For $a \leq a_1$ such that $g(\mu(2 \theta),a) > 1$,
$\theta_a^+ < 2 \theta$.
Hence by (\ref{gc}),
\begin{eqnarray} \label{bJ2}
b_J(\theta) & \geq & f_{\theta}(\mu_a^+,a) \\ \nonumber
& = & a[\theta \mu_a^+-\psi(\theta)]+\tfrac{1}{2}[1-g(\mu_a^+,a)] \\ \nonumber
& = & a[\theta \mu_a^+-\psi(\theta)]+\tfrac{1}{2} \{ 1-a[\theta_a^+ \mu_a^+-\psi(\theta_a^+)]-J(a) \},
\end{eqnarray}
and
\begin{eqnarray} \label{mk2}
m_k & = & E_{\theta_a^+}(e^{(2 \theta-\theta_a^+) Y_1 - 2k \psi(\theta)+k \psi(\theta_a^+)}
{\bf 1}_{\{ Y_1 \leq k \mu_{1n}^* \}}|K_{1n}=k) \\ \nonumber
& \leq & n^{a[(2 \theta-\theta_a^+) \mu_{1n}^*-2 \psi(\theta)+2 \psi(\theta_a^+)]} \\ \nonumber
& \leq & n^{\delta(\theta)+2b_J(\theta)-1+J(a)}.
\end{eqnarray}

Since $b_J(\theta)=\beta-\delta(\theta)$,
by Lemma \ref{pre} and (\ref{mk0})--(\ref{mk2}),
\begin{eqnarray} \label{4.22a}
E_1 \ell_n & = & n E_1 \ell_{1n} \\ \nonumber
& \leq & n^{-\delta(\theta)}
\sum_{k \leq a_1 \log n} e^{J(\frac{k}{\log n})} P(K_{1n}=k) \rightarrow 0,
\end{eqnarray}
and Lemma \ref{lem7}(b) holds.
\end{proof}

\begin{proof}[Proof of Lemma \ref{lem7}(c)]
Since $|\log(1+x)|^2 \leq 4x^2$ for $x \geq -\tfrac{1}{2}$,
\begin{eqnarray} \label{v1}
{\rm Var}_1 \ell_n & = & n {\rm Var}_1 \ell_{1n} \\ \nonumber
& \leq & n E_1 \ell_{1n}^2 \\ \nonumber
& \leq & 4n^{1-2 \beta} \sum_{k \leq a_1 \log n} m_k P(K_{1n}=k) + 4n^{1-3 \beta} \sum_{k \leq a_1 \log n} \zeta_k P(K_{1n}=k), 
\end{eqnarray} 
where 
\begin{eqnarray*}
m_k & = & E_0(e^{2 \theta Y_1-2k \psi(\theta)} {\bf 1}_{\{ Y_1 \leq k \mu_{1n}^* \}}|K_{1n}=k), \cr
& = & E_{\theta} (e^{\theta Y_1-k \psi(\theta)} {\bf 1}_{\{ Y_1 \leq k \mu_{1n}^* \}}|K_{1n}=k), \cr
\zeta_k & = & E_{\theta} (e^{2 \theta Y_1-2 k \psi(\theta)} {\bf 1}_{\{ Y_1 \leq k \mu_{1n}^* \}}|K_{1n}=k).
\end{eqnarray*}
It was shown in the proof of Lemma \ref{lem7}(b),
see (\ref{mk0}) and (\ref{4.22a}),
that
\begin{equation} \label{v2}
n^{1-2 \beta} \sum_{k \leq a_1 \log n} m_k P(K_{1n}=k) \rightarrow 0.
\end{equation}

Let $a=\tfrac{k}{\log n}$.
For $a \leq a_1$ such that $\theta \leq \tfrac{\theta_a^+}{3}$,
apply the inequality
\begin{equation} \label{v3}
\zeta_k \leq E_0(e^{3 \theta Y_1-3k \psi(\theta)}|K_{1n}=k) = e^{a[\psi(3 \theta)-3 \psi(\theta)]}.
\end{equation}
For $a \leq a_1$ such that $\theta > \tfrac{\theta_a^+}{3}$,
by (\ref{bJ2}),
\begin{eqnarray} \label{v4}
\zeta_k & = & E_{\theta_a^+}(e^{(3 \theta-\theta_a^+) Y_1-3k \psi(\theta) + k \psi(\theta_a^+)}
{\bf 1}_{\{ Y_1 \leq k \mu_{1n}^* \}}|K_{1n}=k) \\ \nonumber
& \leq & n^{a[(3 \theta-\theta_a^+) \mu_{1n}^*-3 \psi(\theta)+\psi(\theta_a^+)]} \\ \nonumber
& \leq & n^{3 b_J(\theta)+\frac{3 \delta(\theta)}{2}-1+J(a)-\frac{1}{2}[1-g(\mu_a^+,a)]} \\ \nonumber
& = & n^{3 b_J(\theta)+\frac{3 \delta(\theta)}{2}-1+J(a)}.
\end{eqnarray}

It follows from (\ref{bJ1}) and Lemma \ref{lemX} below that for $0 < \theta \leq \tfrac{\theta_a^+}{3}$,
\begin{eqnarray*}
a[\psi(3 \theta)-3 \psi(\theta)] & \leq & -1+3 f_{\theta}(\mu(2 \theta),a) + J(a) \cr
& \leq & -1+3b_J(\theta)+J(a).
\end{eqnarray*}

Since $b_J(\theta)=\beta-\delta(\theta)$,
by Lemma \ref{pre}, (\ref{v3}) and (\ref{v4}),
\begin{eqnarray*}
& & n^{1-3 \beta} \sum_{k \leq a_1 \log n} \zeta_k P(K_{1n}=k) \cr
& \leq & n^{-\frac{3 \delta(\theta)}{2}}
\sum_{k \leq a_1 \log n} e^{J(\frac{k}{\log n})} P(K_{1n}=k) \rightarrow 0,
\end{eqnarray*}
and ${\rm Var}_1 \ell_n \rightarrow 0$ follows from (\ref{v1}) and (\ref{v2}). 
\end{proof}

\begin{lem} \label{lemX}
For $0 < \theta \leq \tfrac{\theta_a^+}{3}$,
$$2\{ 1+a[\psi(3 \theta)-3 \psi(\theta)] \} \leq 3(1+a[\psi(2 \theta)-2 \psi(\theta)])-J(a).
$$
\end{lem}

\begin{proof}
Since $\mu(\cdot)=\psi'(\cdot)$ is increasing,
\begin{eqnarray*}
& & 3(1+a[\psi(2 \theta)-2 \psi(\theta)])-2-2a[\psi(3 \theta)-3 \psi(\theta)] \cr
& = & 1+3a \psi(2 \theta)-2a \psi(3 \theta) \cr
& = & 1+3a[\psi(2 \theta)-\psi(3 \theta)]+a \psi(3 \theta) \cr
& \geq & 1-3a \theta \psi'(3 \theta)+a \psi(3 \theta) \cr
& = & 1-a[3 \theta \mu(3 \theta)-\psi(3 \theta)] \cr
& = & 1-a I(\mu(3 \theta)) \cr
& \geq & 1-a I(\mu_a^+) \cr
& = & J(a),
\end{eqnarray*}
and Lemma \ref{lemX} holds.
\end{proof}

\section{Proof of Theorem \ref{thm3}}

Define 
\begin{eqnarray} \label{bHa}
b_H(\theta) & = & \max_{(\nu,a): a I(\nu) \leq 1} f_{H \theta}(\nu,a), \\ \label{fH}
\mbox{where } f_{H \theta}(\nu,a) &  =& a[\theta \nu-\psi(\theta)]+\tfrac{1}{2}[1-g_H(\nu,a)], \\ \label{gH}
g_H(\nu,a) & = & a I(\nu)+2 J(a).
\end{eqnarray}

\begin{thm} \label{thm3}
Consider the sparse mixture problem in Theorem \ref{thm2}.
The test statistic HC$_n$ is asymptotically powerless when $\beta > b_H(\theta)$
and asymptotically powerful when $\beta < b_H(\theta)$.
\end{thm}

Let $\delta(\theta)=|b_H(\theta)-\beta|$.

\subsection{Asymptotically powerful}

We show here that HC$_n$ is asymptotically powerful when $b_H(\theta) > \beta$.

\begin{proof}[Proof of Theorem 3 for $b_H(\theta) > \beta$]
By (\ref{bHa}),
there exists $(a,\nu)$ such that
\begin{equation} \label{bH}
a I(\nu) \leq 1, \qquad b_H(\theta) = f_{H \theta}(\nu,a).
\end{equation}

Consider $k = \lceil a \log n \rceil$ and let 
\begin{eqnarray*}
p_n^* & = & 2 P_0(Y_1 \geq k \nu|K_{1n}=k), \cr
q_n & = & P_{\theta}(Y_1 \geq k \nu|K_{1n}=k).
\end{eqnarray*}
Let $n_k^*= \# \{ i: K_{in}=k \}$.
By Lemma \ref{lem1}, it suffices to show that
\begin{eqnarray} \label{4.1}
P_1({\rm HC}_n(p_n^*) \geq \log n) & \rightarrow & 1, \\ \label{4.2}
\mbox{where }{\rm HC}_n(p) & =& \tfrac{\# \{ i: p_i \leq p \}-np}{\sqrt{n p(1-p)}}.
\end{eqnarray}

By Chebyshev's inequality,
under $P_1$,
\begin{equation} \label{4.3}
\# \{ i: p_i \leq p_n^* \} \geq n p_n^* + O_p(\sqrt{n p_n^*})+[1+o_p(1)] \epsilon_n n_k^* q_n,
\end{equation}
with $\epsilon_n = n^{-\beta}$,
provided $\epsilon_n n_k^* q_n \stackrel{p}{\rightarrow} \infty$.
By (\ref{4.2}) and (\ref{4.3}),
\begin{equation} \label{4.4}
{\rm HC}_n(p_n^*) \geq [1+o_p(1)] \tfrac{\epsilon_n n_k^* q_n}{\sqrt{np_n^*}} + O_p(1).
\end{equation}

By (A3), Lemma \ref{pre} and Cram\'{e}r's Theorem,
\begin{eqnarray} \label{4.5}
n_k^* & = & n^{1-J(a)+o_p(1)}, \\ \label{4.6}
p_n^* & = & n^{-aI(\nu)+o(1)}, \\ \label{4.7}
q_n & = & n^{-a[I(\nu)-\theta \nu+\psi(\theta)]+o(1)}.
\end{eqnarray}

By (\ref{bH}) and (\ref{4.5})--(\ref{4.7}),
\begin{eqnarray*} 
\tfrac{\epsilon_n n_k^* q_n}{\sqrt{n p_n^*}} & = & n^{-\beta+\frac{1}{2}-J(a)-\frac{1}{2}a[I(\nu)-2 \theta \nu+2 \psi(\theta)]
+o_p(1)} \cr
& = & n^{b_H(\theta)-\beta+o_p(1)} = n^{\delta(\theta)+o_p(1)}, \cr
\epsilon_n n_k^* q_n & = & n^{-\beta+1-J(a)-a[I(\nu)-\theta \nu+\psi(\theta)]+o_p(1)} \cr
& = & n^{-\beta+b_H(\theta)+\frac{1}{2}[1-aI(\nu)]+o_p(1)} \cr
& = & n^{\delta(\theta)+\frac{1}{2}[1-aI(\nu)]+o_p(1)} \stackrel{p}{\rightarrow} \infty,
\end{eqnarray*}
and (\ref{4.1}) follows from (\ref{4.4}).
\end{proof}

\subsection{Asymptotically powerless}

By Jaeschke (1979),
see also Shorack and Wellner (2009) Chapter 16,
under $H_0$,
\begin{equation} \label{EV2}
b_n {\rm HC}_n-c_n \Rightarrow E_v^2,
\end{equation}
where $b_n = \sqrt{2 \log \log n}$,
$c_n = 2 \log \log n+\tfrac{\log \log \log n}{2}-\tfrac{\log(4 \pi)}{2}$
and $E_v(x) = \exp(-e^{-x})$
is the standardized extreme-value distribution function.

It follows from (\ref{EV2}) and a coupling argument that {\rm HC}$_n$ is asymptotically powerless if
\begin{equation} \label{conv}
\sup_{0 < p \leq \frac{1}{2}} \tfrac{b_n C_n(p)}{\sqrt{np}} \stackrel{p}{\rightarrow} 0,
\end{equation}
where $C_n(p) = \# \{ i: Y_{in} \sim F_{\theta}^{K_{in}}, p_i \leq p \}$.

\begin{proof}[Proof of Theorem \ref{thm3} when $b_H(\theta) < \beta$]
We show (\ref{conv}) using a discretization argument.
Let $m$ be a positive integer satisfying
\begin{equation} \label{bcond}
\tfrac{1}{m} < \delta(\theta) \mbox{ and } \beta > \tfrac{1}{2}+\tfrac{1}{m}.
\end{equation}
Let $M_1=[n^{-\frac{2}{m}},\tfrac{1}{2}]$,
$$M_j = [n^{-\frac{j+1}{m}},n^{-\frac{j}{m}})
\mbox{ for } 2 \leq j \leq m-1,
$$
and $M_m=(0,n^{-1}]$.
We show (\ref{conv}) by showing that 
\begin{equation} \label{show2}
\sup_{p \in M_j} \tfrac{b_n C_n(p)}{\sqrt{np}} \stackrel{p}{\rightarrow} 0, \quad 1 \leq j \leq m.
\end{equation}

Express $C_n(n^{-\frac{j}{m}})$ more simply as $C_{nj}$.
Fix $2 \leq j \leq m$ and for each $k=a \log n$ with $a \leq a_1$ let $\nu \geq \mu(0)$ be such that
$a I(\nu)=\frac{j-1}{m}$.
By (\ref{bHa})--(\ref{gH}),
\begin{eqnarray*}
P_{\theta}(Y_{in} \geq k \nu|K_{in}=k) & \leq & e^{k[\theta \nu-\psi(\theta)-I(\nu)+o(1)]} \cr
& = & n^{a[\theta \nu-\psi(\theta)-I(\nu)]+o(1)} \cr
& \leq & n^{b_H(\theta)-\frac{1}{2}-\frac{1}{2} a I(\nu)+J(a)+o(1)}, \cr
2 P_0(Y_{in} \geq k \nu| K_{in}=k) & = & e^{-k[I(\nu)+o(1)]} \cr
& \geq & n^{-\frac{j}{m}},
\end{eqnarray*}
uniformly over $a \leq a_1$.
Since
$$n \epsilon_n P_{\theta}(K_{in} \geq a_1 \log n) = n^{-\beta+o(1)},
$$
with $\epsilon_n=n^{-\beta}$,
by Lemma \ref{pre},
\begin{eqnarray*}
C_{nj} & \leq & [2+o_p(1)] n \epsilon_n \sum_{k=1}^{\infty} P_{\theta}(Y_{in} \geq k \nu|K_{in}=k) P(K_{in}=k) \cr
& \leq & n^{b_H(\theta)-\beta+\frac{1}{2}-\frac{1}{2}(\frac{j-1}{m})+o_p(1)} \sum_{k \leq a_1 \log n} e^{J(a)} P(K_{in}=k) \cr
& = & n^{\frac{1}{2}-\delta(\theta)-\frac{1}{2}(\frac{j-1}{m})+o_p(1)},
\end{eqnarray*}
and so by (\ref{bcond}),
for $2 \leq j \leq m-1$,
\begin{equation} \label{Q1}
\sup_{p \in M_j} \tfrac{b_n C_n(p)}{\sqrt{np}} \leq b_n C_{nj} n^{-\frac{1}{2}-\frac{1}{2}(\frac{j+1}{m})}
\leq n^{-\delta(\theta)+\frac{1}{m}+o_p(1)} \stackrel{p}{\rightarrow} 0.
\end{equation}
Moreover
$$C_{nm} \leq n^{-\delta(\theta)+\frac{1}{2m}+o_p(1)} \stackrel{p}{\rightarrow} 0,
$$
which implies $P_1(C_{nm}=0) \rightarrow 1$ and thus 
\begin{equation} \label{Q2}
\sup_{p \in M_m} \tfrac{b_n C_n(p)}{\sqrt{np}} \stackrel{p}{\rightarrow} 0.
\end{equation}

In addition,
$$C_{n1} \leq [1+o_p(1)] n \epsilon_n = n^{1-\beta+o_p(1)},
$$
and by (\ref{bcond}),
this implies 
\begin{equation} \label{Q3}
\sup_{p \in M_1} \tfrac{b_n C_n(p)}{\sqrt{np}} \leq n^{\frac{1}{2}-\beta+\frac{1}{m}+o_p(1)} 
\stackrel{p}{\rightarrow} 0.
\end{equation}
We conclude (\ref{show2}) from (\ref{Q1})--(\ref{Q3}). 
\end{proof}

\section{Proof of Theorem \ref{thm4}}

Let $p_{(1)}$ be the smallest p-value and let $p_{(1)k}$ be the smallest p-value among
$\{ p_i: K_{in} \geq k \}$.
Let $n_k = \# \{ i: K_{in} \geq k \}$.
The Bonferroni test statistic is $p_{(1)}$ whereas the rank-adjustment test statistic is
$$R_n = \min_{k \geq 1} n_k p_{(1)k}.
$$

Let $(\theta_H^*,a_{H \theta}^*)$ be such that
\begin{eqnarray} \label{6B}
b_B(\theta) & = & f_{H \theta}(\mu(\theta_H^*),a_{H \theta}^*), \\ \nonumber
& = & a[\theta \mu(\theta_H^*)-\psi(\theta)]-J(a_{H \theta}^*), \\ \label{6.2}
a_{H \theta}^* I(\mu(\theta_H^*)) & = & 1,
\end{eqnarray}
with $\tfrac{\theta_H^*}{\theta} \geq 1$.

Recall from (\ref{bf}) that $g(\nu,a)=a I(\nu)+J(a)$.
Let $(\theta^*,a_{\theta}^*)$ be such that
\begin{eqnarray} \label{6R}
b_R(\theta) & = & \max_{(\nu,a): g(\nu,a)=1} a[\theta \nu-\psi(\theta)] \\ \nonumber
& = & f_{\theta}(\mu(\theta^*),a_{\theta}^*), \\ \label{6.4}
g(\mu(\theta^*),a_{\theta}^*) &  = & 1,
\end{eqnarray}
with $\tfrac{\theta^*}{\theta} \geq 1$.

\begin{thm} \label{thm4}
Consider the sparse mixture problem in Theorem \ref{thm2}.
The Bonferroni test is asymptotically powerless when $\beta > b_B(\theta)$ and asymptotically powerful when
$\beta < b_B(\theta)$.
The rank-adjustment test is asymptotically powerless when $\beta > b_R(\theta)$ and
asymptotically powerful when $\beta < b_R(\theta)$.
\end{thm}

\subsection{Asymptotically powerful Bonferroni test}

Under $H_0$,
\begin{equation} \label{weak}
n p_{(1)} \Rightarrow {\rm Exp}(1),
\end{equation}
the exponential distribution with mean 1.
Hence to show that the Bonferroni test statistic is asymptotically powerful, 
it suffices to show that there exists $p_n^*=o(n^{-1})$ such that
\begin{equation} \label{strong}
P_1(p_{(1)} \leq p_n^*) \rightarrow 1.
\end{equation}

\begin{proof}[Proof of Theorem \ref{thm4} for $\beta < b_B(\theta)$]
Let $k = \lceil a_{H \theta}^* \log n \rceil$ and
$$p_n^* = 2P_0(Y_{in} \geq K_{in} \mu(\theta_H^*)|K_{in}=k).
$$
It follows from Cram\'{e}r's Theorem and (\ref{6.2}) that $p_n^*=o(n^{-1})$.
Let $B_k  = \{ i: K_{in}=k, Y_{in} \sim F_{\theta}^k \}$.
It follows from (A3) that
\begin{equation} \label{nkstar}
E(\# B_k) = n^{1-J(a_{H \theta}^*)-\beta+o(1)}.
\end{equation}
Let
$$q_n = P_{\theta}(Y_{in} \geq K_{in} \mu(\theta_H^*)|K_{in}=k).
$$
By Cram\'{e}r's Theorem, 
\begin{eqnarray*} 
q_n & = & e^{-k[I(\mu(\theta_H^*))-\theta \mu(\theta_H^*)+ \psi(\theta_H^*)+o(1)]} \\ \nonumber
& = & n^{-a_{H \theta}^* I(\mu(\theta_H^*)) + a_{H \theta}^*[\theta \mu(\theta_H^*)-\psi(\theta)]+o(1)} \\ \nonumber
& = & n^{-1+b_B(\theta)+J(a_{H \theta}^*)+o(1)}.
\end{eqnarray*}
Hence by (\ref{nkstar}),
$$q_n E(\# B_k) = n^{b_B(\theta)-\beta+o(1)} \rightarrow 0,
$$
and therefore
$$P_1(p_{(1)} > p_n^*) \leq E[(1-q_n)^{\#B_k}] \leq E(e^{-q_n (\#B_k)}) \rightarrow 0,
$$
and so (\ref{strong}) holds. 
\end{proof}

\subsection{Asymptotically powerless Bonferroni test}

Consider $\beta = b_B(\theta)+\delta(\theta)$ for some $\delta(\theta)>0$.
Let $\wtd p_{(1)} = \min \{ p_i: Y_{in} \sim F_{\theta}^{K_{in}} \}$.
To show that the Bonferroni test is asymptotically powerless,
by (\ref{weak}),
it suffices to show that
\begin{equation} \label{show6.2}
P_1(n \wtd p_{(1)} \leq \log n) \rightarrow 0.
\end{equation}

\begin{proof}[Proof of Theorem \ref{thm4} for $b_B(\theta) < \beta$]
For each $k = a \log n$ with $0 < a \leq a_1$, 
let $\nu$ be such that $a I(\nu) = 1-\tfrac{\delta(\theta)}{2}$.
By Cram\'{e}r's Theorem and Lemma \ref{pre},
for $1 \leq k \leq a_1 \log n$ with $n$ large,
\begin{eqnarray} \label{4a1}
p_1^* := 2 P_0(Y_{1n} \geq k \nu| K_{1n}=k) & = & n^{\frac{\delta(\theta)}{2}-1+o(1)} \geq n^{-1} \log n, \\ \label{4a2}
n^{1-\beta} P(K_{1n} \geq a_1 \log n) & = & n^{-\beta-J(a_1)+o(1)} = o(1).
\end{eqnarray}
For $1 \leq k \leq a_1 \log n$,
\begin{eqnarray} \label{4a3}
P_{\theta}(Y_{1n} \geq k \nu|K_{1n}=k) & \leq & e^{-k[I(\nu)-\theta \nu+\psi(\theta)]} \\ \nonumber
& \leq & n^{-1+b_B(\theta)+\frac{\delta(\theta)}{2}+J(a)+o(1)}.
\end{eqnarray}
By (\ref{4a1})--(\ref{4a3}) and Lemma \ref{pre},
\begin{eqnarray*}
P_1(n \wtd p_{(1)} \leq \log n) & \leq & n^{1-\beta} P_{\theta}(p_1 \leq p_1^*, K_{1n} \leq a_1 \log n)+o(1) \cr
& \leq & 2n^{1-\beta} \sum_{k \leq a_1 \log n} P_{\theta}(Y_{1n} \geq k \nu|K_{1n}=k)
P(K_{1n}=k) + o(1) \cr
& \leq & 2 n^{-\frac{\delta(\theta)}{2}+o(1)} \sum_{k \leq a_1 \log n} e^{J(a)} P(K_{1n}=k) +o(1) \cr
& \rightarrow&  0,
\end{eqnarray*}
and (\ref{show6.2}) holds.
\end{proof}

\subsection{Asymptotically powerful rank-adjustment test}

The rank-adjustment test
$$R_n = \min_{i \geq 1} r_i p_i = \min_{k \geq 1} n_k p_{(1)k},
$$
where $r_i = \# \{ j: K_{jn} \geq K_{in} \}$ and $n_k = \# \{ i: K_{in} \geq k \}$.
It follows from a Bonferroni argument that
$$P_0(R_n \leq \tfrac{1}{(\log n)^2}) \leq \sum_{i=1}^n P_0(p_i \leq \tfrac{1}{r_i (\log n)^2}) 
= \tfrac{1}{(\log n)^2} \sum_{i=1}^n \tfrac{1}{r_i} \rightarrow 0.
$$
Hence to show that the rank-adjustment test is asymptotically powerful,
it suffices to show that
\begin{equation} \label{PRn}
P_1(R_n \leq \tfrac{1}{(\log n)^2}) \rightarrow 1.
\end{equation}
Consider  $b_R(\theta)=\beta+\delta(\theta)$ for some $\delta(\theta)>0$.

\begin{proof}[Proof of Theorem \ref{thm4} for $\beta < b_R(\theta)$]
Let $(\theta^*,a_{\theta}^*)$ be such that
\begin{eqnarray} \label{4b1}
b_R(\theta) & = & a_{\theta}^*[\theta \mu(\theta^*)-\psi(\theta)], \\ \label{4b2}
g(\mu(\theta^*),a_{\theta}^*) & = & a_{\theta}^* I(\mu(\theta^*))+J(a_{\theta}^*)=1,
\end{eqnarray}
with $\tfrac{\theta^*}{\theta} \geq 1$.

Let $k = \lceil a \log n \rceil$,
where 
$$a = a_{\theta}^* + \eta, 
$$
for some $\eta>0$ to be further specified.

Let $B_k = \{ i: K_{in}=k, Y_{in} \sim F_{\theta}^k \}$.
It follows from (A3) and Lemma \ref{pre} that
\begin{eqnarray} \label{4b3} 
E(\# B_k) & = & n^{1-J(a)-\beta+o(1)}, \\ \label{4b4}
En_k & = & n^{1-J(a)+o(1)}.
\end{eqnarray}
By Cram\'{e}r's Theorem, (\ref{4b1}) and (\ref{4b2}), 
\begin{eqnarray} \label{4b5}
p_n^* & := & 2P_0(Y_{in} \geq k \mu(\theta^*)|K_{in}=k) \\ \nonumber
& \leq & e^{-k I(\mu(\theta^*))} \\ \nonumber
& \leq & n^{J(a_{\theta}^*)-1}, \\ \label{4b6}
q_n & := & P_{\theta}(Y_{in} \geq k \mu(\theta^*)|K_{in}=k) \\ \nonumber
& = & e^{-k[I(\mu(\theta^*))-\theta \mu(\theta^*)+\psi(\theta)+o(1)]} \\ \nonumber
& \geq & n^{J(a)-1+b_R(\theta)-\frac{\delta(\theta)}{2}},
\end{eqnarray}
for $\eta > 0$ small.

By (\ref{4b3})--(\ref{4b6}), 
\begin{eqnarray*} 
P_1(n_k p_n^* \leq \tfrac{1}{(\log n)^2}) & \rightarrow & 1, \cr
P_1(p_{(1)k} > p_n^*) & \leq & E[(1-q_n)^{\# B_k}] \cr
& \leq & E(e^{-q_n(\# B_k)}) \cr
& \leq & \exp(-n^{b_R(\theta)-\beta-\frac{\delta(\theta)}{2}+o(1)}) \cr
& = & \exp(-n^{\frac{\delta(\theta)}{2}+o(1)})
\rightarrow 0,
\end{eqnarray*}
and (\ref{PRn}) holds.
\end{proof}

\subsection{Asymptotically powerless rank-adjustment test}

Let $K^*=\max_i K_{in}$ and $r^* = \min_i r_i$.
Since
\begin{eqnarray*}
P_0(\min_i r_i p_i > \log n) & \leq & P_0(\min_{i: K_{in}=K^*} p_i > \tfrac{\log n}{r^*}) \cr
& = & E_0[(1-\tfrac{\log n}{r^*})_+^{r^*}] \cr
& \rightarrow & 0,
\end{eqnarray*}
to show that the rank-adjustment test is asymptotically powerless,
it suffices,
by a coupling argument,
to show that
\begin{equation} \label{s0}
P_1(\min_{i \in B_n} r_i p_i \leq \log n) \rightarrow 0,
\end{equation}
where $B_n = \{ i: Y_{in} \sim F_{\theta}^{K_{in}} \}$.

Consider $\beta=b_R(\theta)+\delta(\theta)$ for some $\delta(\theta)>0$.

\begin{proof}[Proof of Theorem \ref{thm4} for $b_R(\theta)< \beta$]
Let $a_{\delta}>0$ be such that $J(a_{\delta})=1-\tfrac{\delta(\theta)}{2}$.
It follows from (A3) that
\begin{eqnarray} \label{s1}
& & P_1(K_{in} > a_{\delta} \log n \mbox{ for some } i \in B_n) \\ \nonumber
& \leq & n^{1-\beta} P_{\theta}(K_{1n} > a_{\delta} \log n) \\ \nonumber
& = & n^{1-\beta-J(a_{\delta})+o(1)} \rightarrow 0.
\end{eqnarray}

Consider $0 < a \leq a_{\delta}$ and let $\nu_a \geq \mu(0)$ be such that 
\begin{equation} \label{aIJ}
a I(\nu_a)+J(a)=1-\tfrac{\delta(\theta)}{2}.
\end{equation}
For $k=a \log n$,
\begin{eqnarray} \label{s2}
p_k^* & := & 2 P_0(K_{1n} \geq k \nu_a|K_{1n}=k) \\ \nonumber
& = & e^{-k[I(\nu_a)+o(1)]} \\ \nonumber
& = & n^{J(a)+\frac{\delta(\theta)}{2}-1+o(1)}.
\end{eqnarray}

Partition $(a_0,a_{\delta}]$ into $m$ sub-intervals of equal width $M_j=(a^{(j-1)},a^{(j)}]$,
for some large $m$,
such that $a^{(0)}=a_0$ and $a^{(m)}=a_{\delta}$.
Let $n_j = \# \{ i: K_{in} \geq a^{(j)} \log n \}$.
It follows from (\ref{s2}) and $n_0 = n^{1+o_p(1)}$ that
\begin{equation} \label{s3}
P_1(\min_{k \leq a_0 \log n} p_k^* \leq \tfrac{\log n}{n_0}) \rightarrow 0.
\end{equation}
It follows from (\ref{s2}) and $n_j = n^{1-J(a^{(j)})+o_p(1)}$ that for $m$ large,
\begin{equation} \label{s4}
P_1(\min_{k/\log n \in M_j} p_k^* \leq \tfrac{\log n}{n_j}) \rightarrow 0, \quad 1 \leq j \leq m.
\end{equation}

By Cram\'{e}r's Theorem and (\ref{aIJ}),
\begin{eqnarray} \label{s5} 
& & P_1(p_i \leq p_{K_{in}}^* \mbox{ for some } i \in B_n \mbox{ with } K_{in} \leq a_0 \log n) \\ \nonumber
& \leq & n^{1-\beta-\min_{0 < a \leq a_0} a[I(\nu_a)-\theta \nu_a+\psi(\theta)]} \\ \nonumber
& \leq & n^{\frac{\delta(\theta)}{2}-\beta+b_R(\theta)} \rightarrow 0,
\end{eqnarray}
and that for $1 \leq j \leq m$ with $m$ large,
\begin{eqnarray} \label{s6}
& & P_1(p_i \leq p_{K_{in}}^* \mbox{ for some } i \in B_n \mbox{ with } \tfrac{K_{in}}{\log n} \in M_j) \\ \nonumber
& \leq & n^{1-\beta-J(a^{(j-1)})-\min_{a^{(j-1)} \leq a \leq a^{(j)}} a[I(\nu_a)-\theta \nu_a+\psi(\theta)]+o(1)} \\ \nonumber
& \leq & n^{\frac{\delta(\theta)}{2}-\beta+b_R(\theta)+J(a^{(j)})-J(a^{(j-1)})+o(1)} \\ \nonumber
& = & n^{-\frac{\delta(\theta)}{2}+J(a^{(j)})-J(a^{(j-1)})+o(1)} \rightarrow 0.
\end{eqnarray}

We conclude (\ref{s0}) from (\ref{s1}) and (\ref{s3})--(\ref{s6}).
\end{proof}

\end{document}